\documentclass{amsart}

\usepackage{amsfonts}
\usepackage{amsmath}
\usepackage{amssymb}
\usepackage{setspace}
\usepackage{graphicx}
\usepackage{amscd,amssymb,amsthm}

\usepackage{hyperref}
\usepackage{graphicx}
\usepackage{amsmath, amsthm, amscd, amsfonts, amssymb, graphicx, color}
 \usepackage{url}
\usepackage{enumerate}
 \makeatletter
\let\reftagform@=\tagform@
\def\tagform@#1{\maketag@@@{(\ignorespaces\textcolor{blue}{#1}\unskip\@@italiccorr)}}
\renewcommand{\eqref}[1]{\textup{\reftagform@{\ref{#1}}}}
\makeatother
\usepackage{hyperref}
\hypersetup{colorlinks=true, linkcolor=red, anchorcolor=green,
citecolor=cyan, urlcolor=red, filecolor=magenta, pdftoolbar=true}

\usepackage{epsfig}        
\usepackage{epic,eepic}       

\newtheorem{theorem}{Theorem}
\theoremstyle{plain}

\newtheorem{corollary}{Corollary}

\newtheorem{proposition}{Proposition}
\newtheorem{remark}{Remark}

\numberwithin{equation}{section}

\begin{document}

\title[Generalizations of Guessab--Schmeisser formula]{Generalizations of Guessab--Schmeisser formula via Fink type identity with applications to quadrature rules }

\author[M.W. Alomari]{Mohammad W. Alomari}

\address{Department of Mathematics, Faculty of Science and
Information Technology, Irbid National University, P.O. Box 2600,
Irbid, P.C. 21110, Jordan.} \email{mwomath@gmail.com}

\date{\today}


\begin{abstract}
In this work, an expansion of Guessab--Schmeisser two points
formula for $n$-times differentiable functions via  Fink type
identity is established. Generalization of the main result for
harmonic sequence of polynomials is established. Several bounds of
the presented results are proved. As applications, some quadrature
rules are elaborated and discussed. Error bounds of the presented
quadrature rules via Chebyshev-Gr\"{u}ss type inequalities are
also provided.
\end{abstract}

\maketitle

\section{Introduction}

For a continuous function $f$ defined on $[a,b]$, the integral
mean-value theorem (IMVT) guarantees an $x\in [a,b]$ such that
\begin{align}
f\left({x}\right)=\frac{1}{b-a}\int_a^b{f\left({t}\right)dt}.
\label{eq1.1}
\end{align}
In order to measure the difference between any value of $f$ in
$[a,b]$ and its weighted value, Ostrowski  in his celebrated work
\cite{O} established a very interesting inequality for
differentiable functions with bounded derivatives which in
connection with \eqref{eq1.1}, which reads:
\begin{theorem}
\label{thm1}Let $f:I\subset \mathbb{R}\rightarrow \mathbb{R}$ be a
differentiable function on $I^{\circ },$ the interior of the
interval $I,$
such that $f^{\prime }\in L[a,b]$, where $a,b\in I$ with $a<b$. If $%
\left\| {f^{\prime }}\right\|_{\infty}=\mathop {\sup }\limits_{x
\in \left[ {a,b} \right]} \left| {f^{\prime }(x)}\right| \leq
\infty$. Then, the inequality
\begin{align}
\label{eq1.2}\left\vert {\left( {b - a} \right)f\left( {x}\right)
-\int_{a}^{b}{f\left( u\right) du}}\right\vert \leq\left[
{\frac{\left( {b - a} \right)^2 }{4} +
\left({x-\frac{a+b}{2}}\right)^2} \right] \left\| {f^{\prime}}
\right\|_{\infty},
\end{align}
holds for all $x \in [a,b]$. The constant $\frac{1}{4}$ is the
best possible in the sense that it cannot be replaced by a smaller
ones.
\end{theorem}

In 1976  Milovanovi\'{c} and Pe\v{c}ari\'{c} \cite{Milovanovic}
presented their famous generalization of \eqref{eq1.1}  via Taylor
series, where they proved that:
\begin{align}
\label{eq1.3}\left| {\frac{1}{n}\left( {f\left( x \right) +
\sum\limits_{k = 1}^{n - 1} {F_k \left( x \right)} } \right) -
\frac{1}{{b - a}}\int_a^b {f\left( t \right)dt} } \right| \le
C\left( {n,\infty,x} \right)\left\| {f^{\left( n \right)} }
\right\|_\infty,
\end{align}
such that
\begin{align}
F_k \left( x \right) = \frac{{n - k}}{{n!}}\frac{{f^{\left( {k -
1} \right)} \left( a \right)\left( {x - a} \right)^k  - f^{\left(
{k - 1} \right)} \left( b \right)\left( {x - b} \right)^k }}{{b -
a}}.\label{eq1.4}
\end{align}
\noindent  In fact, Milovanovi\'{c} and Pe\v{c}ari\'{c} proved
 the case that $$ C\left( {n,\infty ,x} \right) =
\frac{{\left( {x - a} \right)^{n + 1}  + \left( {b - x} \right)^{n
+ 1} }}{{\left( {b - a} \right)n\left( {n + 1} \right)!}}.$$

In 1992, Fink studied \eqref{eq1.3} in different point of view, he
introduced a new representation of  real $n$-times differentiable
function whose $n$-th derivative $(n\ge1)$ is absolutely
continuous by combining Taylor series and Peano kernel approach
together. Namely, in \cite{Fink} we find:
\begin{multline}
\frac{1}{n}\left( {f\left( x \right)+ \sum\limits_{k = 1}^{n - 1}
{F_k } } \right) - \frac{1}{{b - a}}\int_a^b {f\left( y \right)dy}
\\
=  \frac{1}{{n!\left({b-a}\right) }}\int_a^b {\left( {x - t}
\right)^{n - 1} p\left( {t,x} \right)f^{\left( n \right)} \left( t
\right)dt},\label{eq1.5}
\end{multline}
for all $x\in\left[{a,b}\right]$, where
\begin{align}
p\left( {t,x} \right) = \left\{ \begin{array}{l}
 t - a,\,\,\,\,\,\,\,\,\,\,\,\,\,\,\,\,\,\, t \in \left[ {a,x} \right] \\
 t - b,\,\,\,\,\,\,\,\,\,\,\,\,\,\,\,\,\,\, t \in \left[ {x,b} \right] \\
 \end{array} \right..\label{eq1.6}
\end{align}
In the same work, Fink proved the following bound of
\eqref{eq1.5}.
\begin{align}
\label{eq1.7}\left| {\frac{1}{n}\left( {f\left( x \right) +
\sum\limits_{k = 1}^{n - 1} {F_k \left( x \right)} } \right) -
\frac{1}{{b - a}}\int_a^b {f\left( t \right)dt} } \right| \le
C\left( {n,p,x} \right)\left\| {f^{\left( n \right)} } \right\|_p
\end{align}
where $\left\|  \cdot  \right\|_r$, $1 \le r \le \infty$ are the
usual Lebesgue norms on $L_r [a, b]$, i.e.,
\begin{align*}
\left\| f \right\|_\infty : = ess\mathop {\sup }\limits_{t \in
\left[ {a,b} \right]} \left| {f\left( t \right)} \right|,
\end{align*}
 and
\begin{align*}
\left\| f \right\|_r : = \left( {\int_a^b {\left| {f\left( t
\right)} \right|^r dt} } \right)^{1/r},\,\,\,1 \le r < \infty,
\end{align*}
such that
\begin{align*}
C\left( {n,p,x} \right) = \frac{{\left[ {\left( {x - a}
\right)^{nq + 1}  + \left( {b - x} \right)^{nq + 1} }
\right]^{1/q} }}{{\left( {b - a} \right)n!}} {\rm{B}}^{1/q} \left(
{\left( {n - 1} \right)q + 1,q + 1} \right),
\end{align*}
for $1 < p \le \infty$, $ {\rm{B}}\left(\cdot,\cdot\right)$ is the
beta function, and for $p=1$
$$C\left( {n,1,x} \right) = \frac{{\left( {n - 1} \right)^{n - 1}
}}{{\left( {b - a} \right)n^n n!}}\max \left\{ {\left( {x - a}
\right)^n ,\left( {b - x} \right)^n } \right\}.$$ All previous
bounds are sharp.

Indeed Fink representation can be considered as the first elegant
work (after Darboux work \cite{Kythe}, p.49) that combines two
different approaches together, so that Fink representation is not
less important than Taylor expansion itself. So that, many authors
were interested to study Fink representation approach, more
detailed and related results can be found in
\cite{Aljinovic1},\cite{Aljinovic2},\cite{Geroge1},\cite{Geroge2}
and \cite{Dedic1}.

In 2002 and the subsequent years after that, the Ostrowski's
inequality entered in a new phase of modifications and
developments. A new inequality of Ostrowski's type was born, where
Guessab and Schmeisser in \cite{Guessab} discussed an inequality
from algebraic and analytic points of view which is in connection
with Ostrowski inequality; called  `\emph{the companion of
Ostrowski's inequality}' as suggested later by Dragomir in
\cite{Dragomir1}. The main part of Guessab--Schmeisser inequality
reads the difference between symmetric values of a real function
$f$ defined on $[a,b]$ and its weighed value, i.e.,
\begin{align*}
\frac{{f\left( x \right) + f\left( {a + b - x} \right)}}{2} -
\frac{1}{{b - a}}\int_a^b {f\left( t \right)dt}, \qquad x \in
\left[{a,\frac{a+b}{2}}\right].
\end{align*}
Namely, in the significant work \cite{Guessab} we find the first
primary result is that:
\begin{theorem}
\label{thm2} Let $f : [a,b] \to \mathbb{R}$ be satisfies the
H\"{o}lder condition of order $r\in (0,1]$. Then for each $x \in
[a,\frac{a + b}{2}]$, the we have the inequality
\begin{multline}
\label{eq1.8} \left| {\frac{{f\left( x \right) + f\left( {a + b -
x} \right)}}{2} - \frac{1}{{b - a}}\int_a^b {f\left( t \right)dt}
} \right|
\\
\le \frac{M}{{b - a}}\frac{{\left( {2x - 2a} \right)^{r + 1}  +
\left( {a + b - 2x} \right)^{r + 1} }}{{2^r \left( {r + 1}
\right)}}.
\end{multline}
This inequality is sharp for each admissible $x$. Equality is
attained if and only if $f=\pm M f_{*}+c$ with $c\in\mathbb{R}$
and
\begin{align*}
f_* \left( t \right) = \left\{ \begin{array}{l}
 \left( {x - t} \right)^r ,\,\,\,\,\,\,\,\,\,\,\,\,\,\,\,\,\,\,\,\,\,\,\,{\rm{if}}\,\,\,a \le t \le x \\
 \left( {t - x} \right)^r ,\,\,\,\,\,\,\,\,\,\,\,\,\,\,\,\,\,\,\,\,\,\,\,{\rm{if}}\,\,\,x \le t \le \frac{{a + b}}{2} \\
 f_* \left( {a + b - t} \right),\,\,\,\,\,\,\,\,\,{\rm{if}}\,\,\,\frac{{a + b}}{2} \le t \le b \\
 \end{array} \right..
\end{align*}
\end{theorem}

In the same work \cite{Guessab}, the authors discussed and
investigated  \eqref{eq1.8} for other type of assumptions. Among
others, a brilliant representation (or identity) of $n$-times
differentiable functions whose $n$-th derivatives are piecewise
continuous was established as follows:
\begin{theorem}
\label{thm3}Let $f$ be a function defined on $[a, b]$ and having
there a piecewise continuous $n$-th derivative. Let $Q_n$ be any
monic polynomial of degree $n$ such that $Q_n\left(t\right)=
\left(-1\right)^n Q_n\left(a+b-t\right)$. Define
\begin{align*}
 K_n \left( t
\right) = \left\{ \begin{array}{l}
 \left( {t - a} \right)^n ,\,\,\,\,\,\,\,\,\,\,\,\,{\rm{if}}\,\,\,\,a \le t \le x \\
  \\
 Q_n \left( t \right),\,\,\,\,\,\,\,\,\,\,\,\,\,\,\,\,{\rm{if}}\,\,\,\,x \le t \le a + b - x \\
  \\
 \left( {t - b} \right)^n ,\,\,\,\,\,\,\,\,\,\,\,\,{\rm{if}}\,\,\,\,a + b - x \le t \le b \\
 \end{array} \right..
\end{align*}
 Then,
\begin{align}
\label{eq1.9}\int_a^b {f\left( t \right)dt}  = \left( {b - a}
\right)\frac{{f\left( x \right) + f\left( {a + b - x} \right)}}{2}
+ E\left( {f;x} \right)
\end{align}
where,
\begin{multline*}
E\left( {f;x} \right) = \sum\limits_{\nu  = 1}^{n - 1} {\left[
{\frac{{\left( {x - a} \right)^{\nu  + 1} }}{{\left( {\nu  + 1}
\right)!}}-\frac{{Q_n^{\left( {n - \nu  - 1} \right)} \left( x
\right)}}{{n!}}} \right]\left[ {f^{\left( \nu  \right)} \left( {a
+ b - x} \right) + \left( { - 1} \right)^\nu  f\left( x \right)}
\right]}
\\
+ \frac{{\left( { - 1} \right)}}{{n!}}\int_a^b {K_n \left( t
\right)f^{\left( n \right)} \left( t \right)dt}.
\end{multline*}
\end{theorem}
This generalization \eqref{eq1.9} can be considered  as a
companion type expansion of Euler--Maclaurin formula that expand
symmetric values of real functions. In this way, families of
various quadrature rules can be presented, as shown -for example-
in \cite{Kovac}. Therefore, since 2002 and after the presentation
of \eqref{eq1.8}, several authors have studied, developed  and
established  new presentations concerning \eqref{eq1.8} using
several approaches and different tools, for this purpose see the
recent survey \cite{Dragomir}.

Far away from this, in the last thirty years  the concept of
harmonic sequence of polynomials or Appell polynomials have been
used at large in numerical integrations and expansions theory of
real functions. Let us recall that, a  sequence of polynomials
$\left\{{P_k\left(t, \cdot\right)}\right\}_{k=0}^{\infty}$
satisfies the Appell condition (see \cite{Appell}) if
$\frac{\partial}{\partial t} P_k \left( {t, \cdot } \right)=
P_{k-1} \left( {t, \cdot } \right)$ $(\forall k\ge1)$ with $P_0
\left( {t, \cdot } \right)=1$, for all well-defined order pair
$\left( {t, \cdot } \right)$. A slightly different definition was
considered in \cite{matic}.

In 2003, motivated by work of Mati\'{c} et. al. \cite{matic},
Dedi\'{c} et. al. in \cite{Dedic1}, introduced the following smart
generalization of Ostrowski's inequality via harmonic sequence of
polynomials:
\begin{multline}
\label{eq1.10}\frac{1}{n}\left[ {f\left( x \right) +
\sum\limits_{k = 1}^{n - 1} {\left( { - 1} \right)^k P_k \left( x
\right)f^{\left( k \right)} \left( x \right)}  + \sum\limits_{k =
1}^{n - 1} {\widetilde{F_k }\left( {a,b} \right)} } \right]
\\
= \frac{{\left( { - 1} \right)^{n - 1} }}{{\left( {b - a}
\right)n}}\int_a^b {P_{n - 1} \left( t \right)p\left( {t,x}
\right)f^{\left( n \right)} \left( t \right)dt},
\end{multline}
where $P_k$ is a harmonic sequence of polynomials satisfies that
$P^{\prime}_k=P_{k-1}$ with $P_0=1$,
\begin{align}
\label{eq1.11}\widetilde{F_k }\left( {a,b} \right) = \frac{{\left(
{ - 1} \right)^k \left( {n - k} \right)}}{{b - a}}\left[ {P_k
\left( a \right)f^{\left( {k - 1} \right)} \left( a \right) - P_k
\left( b \right)f^{\left( {k - 1} \right)} \left( b \right)}
\right]
\end{align}
and $p\left( {t,x} \right)$ is given in \eqref{eq1.5}. In
particular, if we take $P_k\left( {t}
\right)=\frac{\left(t-x\right)^k}{k!}$ then we refer to Fink
representation \eqref{eq1.5}.

In 2005, Dragomir \cite{Dragomir1}  proved the following bounds of
the companion of Ostrowski's inequality for absolutely continuous
functions.
\begin{theorem}\label{thm4}
Let $f:I\subset \mathbb{R}\rightarrow \mathbb{R}$ be an absolutely
continuous function on $[a,b]$. Then we have the inequalities
\begin{multline}
\label{eq1.12}\left| {\frac{{f\left( x \right) + f\left( {a + b -
x} \right)}}{2} - \frac{1}{{b - a}}\int_a^b {f\left( t \right)dt}
} \right| \\\le \left\{ \begin{array}{l}
 \left[ {\frac{1}{8} + 2\left( {\frac{{x - {\textstyle{{3a + b} \over 4}}}}{{b - a}}} \right)^2 } \right]\left( {b - a} \right)\left\| {f'} \right\|_\infty  ,\,\,\,\,\,\,\,\,\,f' \in L_\infty  \left[ {a,b} \right] \\
  \\
 \frac{{2^{1/q} }}{{\left( {q + 1} \right)^{1/q} }}\left[ {\left( {\frac{{x - a}}{{b - a}}} \right)^{q + 1}  - \left( {\frac{{{\textstyle{{a + b} \over 2}} - x}}{{b - a}}} \right)^{q + 1} } \right]^{1/q} \left( {b - a} \right)^{1/q} \left\| {f'} \right\|_{\left[ {a,b} \right],p} , \\
 \,\,\,\,\,\,\,\,\,\,\,\,\,\,\,\,\,\,\,\,\,\,\,\,\,\,\,\,\,\,\,\,\,\,\,\,\,\,\,\,\,\,\,\,\,\,\,\,\,\,\,\,\,\,\,\,\,\,\,\,\,\,\,\,\,\,\,\,\,\,\,\,\,\,p > 1,\frac{1}{p} + \frac{1}{q} = 1,\,and\,f' \in L_p \left[ {a,b} \right] \\
 \left[ {\frac{1}{4} + \left| {\frac{{x - {\textstyle{{3a + b} \over 4}}}}{{b - a}}} \right|} \right]\left\| {f'} \right\|_{\left[ {a,b} \right],1}  \\
 \end{array} \right.
\end{multline}
for all $x \in [a,\frac{a + b}{2}]$. The constants $\frac{1}{8}$
and  $\frac{1}{4}$ are the best possible in (\ref{eq1.12}) in the
sense that it cannot be replaced by smaller constants.

\end{theorem}
The author of this paper have took a serious attention to
Guessab--Schmeisser inequality  in the works
\cite{alomari1}--\cite{alomari9}.
 For other related results and
generalizations concerning Ostrowski's inequality and its
applications we refer the reader to
\cite{Geroge1}--\cite{Cerone5}, \cite{Dedic2},
\cite{Dragomir1}--\cite{Dragomir3}, \cite{Kechriniotis},
\cite{Liu1}, \cite{Sofo} and \cite{Nenad}.

In the last fifteen years, constructions of quadrature rules using
expansion of an arbitrary function in Bernoulli polynomials and
Euler--Maclaurin's type formulae have been established, improved
and investigated. These approaches permit many researchers to work
effectively in the area of numerical integration where several
error approximations of various quadrature rules presented with
high degree of exactness. Mainly,  works of Dedi\'{c} et al.
\cite{Dedic1}--\cite{Dedic2}, Aljinovi\'{c} et al.
\cite{Aljinovic1}, \cite{Aljinovic2}, Kova\'{c} et al.
\cite{Kovac} and others, received positive responses and good
interactions from other focused researchers. Among others,
Franji\'{c} et al. in several works (such as \cite{F1}--\cite{F5})
constructed several Newton-Cotes and Gauss quadrature type rules
using a certain expansion of real functions in Bernoulli
polynomials or Euler--Maclaurin's type formulae.

Unfortuentaley, the expansions \eqref{eq1.5}, \eqref{eq1.9} and
\eqref{eq1.10} have not been  used to construct quadrature rules
yet. It seems these expansions were abandoned or neglected in
literature because most of authors are still use the classical
Euler--Maclaurin's formula and expansions in Bernoulli
polynomials.

This work has several aims and goals, the first aim is to
generalize Guessab--Schmeisser two points formula for $n$-times
differentiable functions via Fink type identity and provide
several type of bounds for the remainder formula. The second goal,
is to highlight the importance of these expansions and give a
serious attention to their applicable usefulness in constructing
various quadrature rules. The third aim,  is to spotlight the role
of \v{C}eby\v{s}ev functional in integral approximations.

This work is organized as follows: in the next section, a
Guessab--Schmeisser two points formula for $n$-times
differentiable functions via Fink type identity is established.
Bounds for the remainder term of the presented formula are proved.
In section 3, bounds for the remainder term via
Chebyshev-Gr\"{u}ss type inequalities are presented. In section 4,
generalizations of the obtained results to harmonic sequence of
polynomials are given. In section 5, representations of some
quadrature rules are introduced and their errors are explored.

\section{The Results}

\subsection{Guessab--Schmeisser formula via Fink type identity}

\begin{theorem}\label{thm5}
Let $I$ be a real interval, $a,b \in I^{\circ}$ $(a<b)$. Let $f:I
\to \mathbb{R}$ be $n$-times differentiable on $I^{\circ}$ such
that $f^{(n)}$ is absolutely continuous on $I^{\circ}$ with
$\left( {\cdot - t} \right)^{n - 1} S\left( {t,\cdot}
\right)f^{\left( n \right)} \left( t \right)$ is integrable. Then
we have the representation
\begin{multline}
\frac{1}{n}\left( {\frac{{f\left( x \right) + f\left( {a + b - x}
\right)}}{2} + \sum\limits_{k = 1}^{n - 1} {G_k } } \right) -
\frac{1}{{b - a}}\int_a^b {f\left( y \right)dy}
\\
=  \frac{1}{{n!\left({b-a}\right) }}\int_a^b {\left( {x - t}
\right)^{n - 1} S\left( {t,x} \right)f^{\left( n \right)} \left( t
\right)dt},\label{eq2.1}
\end{multline}
for all $x\in\left[{a,\frac{a+b}{2}}\right]$, where
\begin{multline}
G_k:=G_k\left(x\right)  = \frac{{\left( {n - k}
\right)}}{{k!\left( {b - a} \right)}} \cdot \left\{{ \left( {x -
a} \right)^k \left[ {f^{\left( {k - 1} \right)} \left( a \right) +
\left( { - 1} \right)^{k + 1} f^{\left( {k - 1} \right)} \left( b
\right)} \right] }\right.
\\
\left.{+ \left( {1 + \left( { - 1} \right)^{k + 1} } \right)\left(
{\frac{{a + b}}{2} - x} \right)^k f^{\left( {k - 1} \right)}
\left( {\frac{{a + b}}{2}} \right) }\right\}, \label{eq2.2}
\end{multline}
and
\begin{align}
S\left( {t,x} \right) = \left\{ \begin{array}{l}
 t - a,\,\,\,\,\,\,\,\,\,\,\,\,\,\,\,\,\,\, t \in \left[ {a,x} \right] \\
 t - \frac{{a + b}}{2},\,\,\,\,\,\,\,\,\,\,\,t \in \left( {x,a + b - x} \right) \\
 t - b,\,\,\,\,\,\,\,\,\,\,\,\,\,\,\,\,\,\, t \in \left[ {a + b - x,b} \right] \\
 \end{array} \right..
\end{align}

\end{theorem}

\begin{proof}
Fix $x\in \left[a,b\right]$. Starting with Taylor series expansion
for $f$ along $\left[{a,\frac{a+b}{2}}\right]$
\begin{align}
f\left( x \right) = f\left( y \right) + \sum\limits_{k = 1}^{n -
1} {\frac{{f^{\left( k \right)} \left( y \right)}}{{k!}}\left( {x
- y} \right)^k }  + \frac{1}{{\left( {n - 1} \right)!}}\int_y^x
{\left( {x - t} \right)^{n - 1} f^{\left( n \right)} \left( t
\right)dt}.\label{eq2.4}
\end{align}
Integrating with respect to $y$ along
$\left[{a,\frac{a+b}{2}}\right]$, we have
\begin{multline}
\frac{{b - a}}{2}f\left( x \right) = \int_a^{\frac{a+b}{2}}
{f\left( y \right)dy}  + \sum\limits_{k = 1}^{n - 1}
{\frac{1}{{k!}}\int_a^{\frac{a+b}{2}} {\left( {x - y} \right)^k
f^{\left( k \right)} \left( y \right)dy} }
\\
+ \frac{1}{{\left( {n - 1} \right)!}}\int_a^{\frac{a+b}{2}}
{\left( {\int_y^x {\left( {x - t} \right)^{n - 1} f^{\left( n
\right)} \left( t \right)dt} } \right)dy}.\label{eq2.5}
\end{multline}
Also, for $x\in \left[{\frac{a+b}{2},b}\right]$, $f$ has the
representation
\begin{multline}
f\left( {a + b - x} \right) = f\left( y \right) + \sum\limits_{k =
1}^{n - 1} {\frac{{f^{\left( k \right)} \left( y
\right)}}{{k!}}\left( {a + b - x - y} \right)^k }
\\
+ \frac{1}{{\left( {n - 1} \right)!}}\int_y^{a + b - x} {\left( {a
+ b - x - t} \right)^{n - 1} f^{\left( n \right)} \left( t
\right)dt}.\label{eq2.6}
\end{multline}
Integrating with respect to $y$ along
$\left[{\frac{a+b}{2},b}\right]$, we have
\begin{multline}
\frac{{b - a}}{2}f\left( {a + b - x} \right) = \int_{\frac{{a +
b}}{2}}^b {f\left( y \right)dy}  + \sum\limits_{k = 1}^{n - 1}
{\frac{1}{{k!}}\int_{\frac{{a + b}}{2}}^b {\left( {a + b - x - y}
\right)^k f^{\left( k \right)} \left( y \right)dy} }
\\
+ \frac{1}{{\left( {n - 1} \right)!}}\int_{\frac{{a + b}}{2}}^b
{\left( {\int_y^{a + b - x} {\left( {a + b - x - t} \right)^{n -
1} f^{\left( n \right)} \left( t \right)dt} }
\right)dy}\label{eq2.7}
\end{multline}
Adding \eqref{eq2.5} and \eqref{eq2.7}, we get
\begin{align}
&\left( {b - a} \right)\frac{{f\left( x \right) + f\left( {a + b -
x} \right)}}{2}
\nonumber\\
&=\int_a^b {f\left( y \right)}  +  \sum\limits_{k = 1}^{n - 1}
{I_k }\label{eq2.8}
\\
&\qquad+ \frac{1}{{\left( {n - 1} \right)!}}\left[
{\int_a^{{\textstyle{{a + b} \over 2}}} {\left( {\int_y^x {\left(
{x - t} \right)^{n - 1} f^{\left( n \right)} \left( t \right)dt} }
\right)dy} }\right.
\nonumber\\
&\qquad\qquad\left.{+ \int_{{\textstyle{{a + b} \over 2}}}^b
{\left( {\int_y^{a + b - x} {\left( {a + b - x - t} \right)^{n -
1} f^{\left( n \right)} \left( t \right)dt} } \right)dy} }
  \right],\nonumber
\end{align}
where, $I_k  = J_k  + h_k$, $I_0  = \int_a^b {f\left( y
\right)dy}$, $J_k    = \frac{1}{{k!}}\int_a^{\frac{a+b}{2}}
{\left( {x - y} \right)^k f^{\left( k \right)} \left( y
\right)dy}$ and $h_k = \frac{1}{{k!}}\int_{\frac{a+b}{2}}^b
{\left( {a + b - x - y} \right)^k f^{\left( k \right)} \left( y
\right)dy}$  $(k\ge1)$. Therefore, the following recurrence
relations follows using integration by parts formula (see
\cite{Milovanovic}):
\begin{align}
\left( {n - k} \right)\left( {J_k  - J_{k - 1} } \right) =  -
\left( {b - a} \right)D_k ,\,\,\,\,\,\,\,\,\,\,\,\,\,\,\left( {1
\le k \le n - 1} \right), \label{eq2.9}
\end{align}
where,
\begin{align*}
D_k  = \frac{{\left( {n - k} \right)}}{{k!}} \cdot \frac{{\left(
{x - a} \right)^k f^{\left( {k - 1} \right)} \left( a \right) -
\left( {x - \frac{{a + b}}{2}} \right)^k f^{\left( {k - 1}
\right)} \left( {\frac{{a + b}}{2}} \right)}}{{b - a}}.
\end{align*}
Similarly we have
\begin{align}
\left( {n - k} \right)\left( {\ell_k  - \ell_{k - 1} } \right) =
- \left( {b - a} \right)L_k ,\,\,\,\,\,\,\,\,\,\,\,\,\,\,\left( {1
\le k \le n - 1} \right) \label{eq2.10}
\end{align}
where,
\begin{align*}
L_k  = \frac{{\left( {n - k} \right)}}{{k!}} \cdot \frac{{\left(
{\frac{{a + b}}{2} - x} \right)^k f^{\left( {k - 1} \right)}
\left( {\frac{{a + b}}{2}} \right) - \left( {a - x} \right)^k
f^{\left( {k - 1} \right)} \left( b \right)}}{{b - a}}.
\end{align*}
Therefore, by adding \eqref{eq2.9} and \eqref{eq2.10} we get
\begin{align}
\left( {n - k} \right)\left( {I_k  - I_{k - 1} } \right) =  -
\left( {b - a} \right)G_k ,\,\,\,\,\,\,\,\,\,\,\,\,\,\,\left( {1
\le k \le n - 1} \right) \label{eq2.11}
\end{align}
where $G_k= D_k +L_k$.

Summing the terms in \eqref{eq2.11} form $k=1$ up to $k=n-1$,
simplifications lead us to write
\begin{align}
\sum\limits_{k = 1}^{n - 1} {I_k }  =  - \left( {b - a}
\right)\sum\limits_{k = 1}^{n - 1} {G_k }  + \left( {n - 1}
\right)I_0.  \label{eq2.12}
\end{align}
Substituting \eqref{eq2.12} in \eqref{eq2.8} and rearrange the
terms we get that
\begin{multline}
\frac{1}{n}\left( {\frac{{f\left( x \right) + f\left( {a + b - x}
\right)}}{2} + \sum\limits_{k = 1}^{n - 1} {G_k } } \right) -
\frac{1}{{b - a}}\int_a^b {f\left( y \right)dy}
\\
=\frac{1}{{n!\left({b-a}\right) }}\left[ {\int_a^{{\textstyle{{a +
b} \over 2}}} {\left( {\int_y^x {\left( {x - t} \right)^{n - 1}
f^{\left( n \right)} \left( t \right)dt} } \right)dy} }\right.
 \\
 \left.{+ \int_{{\textstyle{{a + b} \over 2}}}^b
{\left( {\int_y^{a + b - x} {\left( {a + b - x - t} \right)^{n -
1} f^{\left( n \right)} \left( t \right)dt} } \right)dy} }
  \right].\label{eq2.13}
\end{multline}
To simplify the right hand side, we write
\begin{align}
\int_a^{\frac{{a + b}}{2}} {dy} \int_y^x {dt}  &= \int_a^x {dy}
\int_y^x {dt}  + \int_x^{\frac{{a + b}}{2}} {dy} \int_y^x {dt}
\nonumber\\
&= \int_a^x {dt} \int_a^t {dy}  - \int_x^{\frac{{a + b}}{2}} {dy}
\int_y^x {dt}
\nonumber\\
&= \int_a^x {dt} \int_a^t {dy}  - \int_x^{\frac{{a + b}}{2}} {dt}
\int_t^{\frac{{a + b}}{2}} {dy}, \label{eq2.14}
\end{align}
and
\begin{align}
\int_{\frac{{a + b}}{2}}^b {dy} \int_y^{a + b - x} {dt}  &=
\int_{\frac{{a + b}}{2}}^{a + b - x} {dy} \int_y^{a + b - x} {dt}
+ \int_{a + b - x}^b {dy} \int_y^{a + b - x} {dt}
\nonumber\\
&= \int_{\frac{{a + b}}{2}}^{a + b - x} {dt} \int_{\frac{{a +
b}}{2}}^t {dy}  - \int_{\frac{{a + b}}{2}}^b {dy} \int_y^{a + b -
x} {dt}
\nonumber\\
&= \int_{\frac{{a + b}}{2}}^{a + b - x} {dt} \int_{\frac{{a +
b}}{2}}^t {dy}  - \int_{a+b-x}^b {dt} \int_t^b {dy}.\label{eq2.15}
\end{align}
Adding \eqref{eq2.14} and \eqref{eq2.15}, we get
\begin{multline}
\int_a^{\frac{{a + b}}{2}} {dy} \int_y^x {dt} +\int_{\frac{{a +
b}}{2}}^b {dy} \int_y^{a + b - x} {dt}
\\
= \int_a^x {dt} \int_a^t {dy}  - \int_x^{\frac{{a + b}}{2}} {dt}
\int_t^{\frac{{a + b}}{2}} {dy} + \int_{\frac{{a + b}}{2}}^{a + b
- x} {dt} \int_{\frac{{a + b}}{2}}^t {dy}  - \int_{a+b-x}^b {dt}
\int_t^b {dy}.\label{eq2.16}
\end{multline}
In viewing\eqref{eq2.16}, the right hand side of  \eqref{eq2.13}
becomes
\begin{multline*}
\frac{1}{{n!\left({b-a}\right) }}\left[ {\int_a^{{\textstyle{{a +
b} \over 2}}} {\left( {\int_y^x {\left( {x - t} \right)^{n - 1}
f^{\left( n \right)} \left( t \right)dt} } \right)dy} }\right.
 \\
 \left.{+ \int_{{\textstyle{{a + b} \over 2}}}^b
{\left( {\int_y^{a + b - x} {\left( {a + b - x - t} \right)^{n -
1} f^{\left( n \right)} \left( t \right)dt} } \right)dy} }
  \right]
\\
=  \frac{1}{{n!\left({b-a}\right) }}\int_a^b {\left( {x - t}
\right)^{n - 1} S\left( {t,x} \right)f^{\left( n \right)} \left( t
\right)dt},
\end{multline*}
where
\begin{align*}
S\left( {t,x} \right) = \left\{ \begin{array}{l}
 t - a,\,\,\,\,\,\,\,\,\,\,\,\,\,\,\,\,\,\, t \in \left[ {a,x} \right] \\
 t - \frac{{a + b}}{2},\,\,\,\,\,\,\,\,\,\,\,t \in \left( {x,a + b - x} \right) \\
 t - b,\,\,\,\,\,\,\,\,\,\,\,\,\,\,\,\,\,\, t \in \left[ {a + b - x,b} \right] \\
 \end{array} \right..
\end{align*}
Thus, the identity \eqref{eq2.13} becomes
\begin{multline}
\frac{1}{n}\left( {\frac{{f\left( x \right) + f\left( {a + b - x}
\right)}}{2} + \sum\limits_{k = 1}^{n - 1} {G_k\left(x\right) } }
\right) - \frac{1}{{b - a}}\int_a^b {f\left( y \right)dy}
\\
=  \frac{1}{{n!\left({b-a}\right) }}\int_a^b {\left( {x - t}
\right)^{n - 1} S\left( {t,x} \right)f^{\left( n \right)} \left( t
\right)dt}
\end{multline}
for all $x \in \left[{a,\frac{a+b}{2}}\right]$.
\end{proof}

\begin{theorem}
\label{thm6}Under the assumptions of Theorem \ref{thm5}. We have
\begin{multline}
\left|{ \frac{1}{n}\left( {\frac{{f\left( x \right) + f\left( {a +
b - x} \right)}}{2} + \sum\limits_{k = 1}^{n - 1} {G_k } } \right)
- \frac{1}{{b - a}}\int_a^b {f\left( y \right)dy} }\right|
\\
\le C\left( {n,p,x} \right)\left\| {f^{\left( n \right)} }
\right\|_p\label{eq2.18}
\end{multline}
holds,  where
\begin{align}
M\left( {n,p,x} \right) =   \left\{ \begin{array}{l}
 \frac{1}{n\cdot n!\left( {b-a} \right)}\left( {\frac{{n - 1}}{n}}
\right)^{n - 1} \left[ {\frac{{b - a}}{4} + \left| {x - \frac{{3a
+ b}}{4}} \right|} \right]^n
,\,\,\,\,\,\,\,\,\,\,\,\,\,\,\,\,\,\,\,\,\,\,\,\,\,\,\,{\rm{if}}\,\,p = 1 \\
\\
 \frac{{2^{1/q} }}{{n!\left( {b - a} \right)}}\left[ {\left( {x -
a} \right)^{nq + 1}  + \left( {\frac{{a + b}}{2} - x} \right)^{nq
+ 1} } \right]^{1/q}
\\
\qquad \qquad\times {\rm{B}}^{\frac{1}{q}} \left( {\left( {n - 1}
\right)q + 1,q + 1} \right),
\,\,\,\,\,\,\,\,\,\,\,\,\,\,\,\,\,\,\,\,{\rm{if}}\,\,  1 < p \le \infty, \,q=\frac{p}{p-1}  \\
 \end{array} \right..
\label{eq2.19}
 \end{align}
The constant $C\left( {n,p,x} \right)$ is the best possible in the
sense that it cannot be replaced by a smaller ones.
\end{theorem}

\begin{proof}
Utilizing the triangle integral inequality on the identity
\eqref{eq2.1} and employing some known norm inequalities we get
\begin{align*}
&\left|{ \frac{1}{n}\left( {\frac{{f\left( x \right) + f\left( {a
+ b - x} \right)}}{2} + \sum\limits_{k = 1}^{n - 1}
{G_k\left(x\right) } } \right) - \frac{1}{{b - a}}\int_a^b
{f\left( y \right)dy} } \right|
\\
&\le \frac{1}{{n!\left({b-a}\right) }}\int_a^b {\left| {x - t}
\right|^{n - 1} \left|{S\left( {t,x} \right)}\right|
\left|{f^{\left( n \right)} \left( t \right)}\right|dt}
\\
&\le \left\{ \begin{array}{l}
 \left\| {f^{\left( n \right)} } \right\|_1 \mathop {\sup }\limits_{a \le t \le b} \left\{ {\left| {x - t} \right|^{n - 1} \left| {k\left( {t,x} \right)} \right|} \right\},\,\,\,\,\,\,\,\,\,\,\,\,\,\,\,\,\,\,\,\,\,\,\,\,\,\,p = 1 \\
 \\
 \left\| {f^{\left( n \right)} } \right\|_p \left( {\int_a^b {\left| {x - t} \right|^{\left( {n - 1} \right)q} \left| {k\left( {t,x} \right)} \right|^q dt} } \right)^{1/q} ,\,\,\,\,\,\,\,\,\,\,1 < p < \infty  \\
 \\
 \left\| {f^{\left( n \right)} } \right\|_\infty  \int_a^b {\left| {x - t} \right|^{n - 1} \left| {k\left( {t,x} \right)} \right|dt} ,\,\,\,\,\,\,\,\,\,\,\,\,\,\,\,\,\,\,\,\,\,\,\,\,\,\,\,\,\,\,\,\,\,p = \infty  \\
 \end{array} \right..
\end{align*}
It is easy to find that for $p=1$, we have
\begin{align*}
\mathop {\sup }\limits_{a \le t \le b} \left\{ {\left| {x - t}
\right|^{n - 1} \left| {S\left( {t,x} \right)} \right|}
\right\}&=\frac{1}{n}\left( {\frac{{n - 1}}{n}} \right)^{n - 1}
\max \left\{ {\left( {x - a} \right)^n ,\left( {\frac{{a + b}}{2}
- x} \right)^n } \right\}
\\
&= \frac{1}{n}\left( {\frac{{n - 1}}{n}} \right)^{n - 1} \left[
{\frac{{b - a}}{4} + \left| {x - \frac{{3a + b}}{4}} \right|}
\right]^n,
\end{align*}
and for $1<p<\infty$, we have
\begin{align*}
&\int_a^b {\left| {x - t} \right|^{\left( {n - 1} \right)q} \left|
{S\left( {t,x} \right)} \right|^q dt}
\\
&= \int_a^x {\left| {x - t} \right|^{\left( {n - 1} \right)q}
\left( {t-a} \right)^q dt} + \int_x^{a+b-x} {\left| {x - t}
\right|^{\left( {n - 1} \right)q} \left| {t-\frac{a+b}{2}}
\right|^q dt}
\\
&\qquad+ \int_{a+b-x}^b {\left| {x - t} \right|^{\left( {n - 1}
\right)q} \left( {b-t} \right)^q dt}
\\
&= 2  \left[ {\left( {x - a} \right)^{nq + 1}  + \left( {\frac{{a
+ b}}{2} - x} \right)^{nq + 1} } \right]  \left(
{\int_{\rm{0}}^{\rm{1}} {\left( {1 - s} \right)^{\left( {n - 1}
\right)q} s^q ds} } \right)
\\
&= 2 \left[ {\left( {x - a} \right)^{nq + 1}  + \left( {\frac{{a +
b}}{2} - x} \right)^{nq + 1} } \right] {\rm{B}} \left( {\left( {n
- 1} \right)q + 1,q + 1} \right)
\end{align*}
where, we use the substitutions $t=\left({1-s}\right)a+s x$,
$t=\left({1-s}\right)x+s\left({a+b-x}\right)$ and
$t=\left({1-s}\right)\left({a+b-x}\right)+sb$; respectively.The
third case, $p=\infty$ holds by setting $p= \infty$ and $q=1$,
i.e.,
\begin{align*}
&\int_a^b {\left| {x - t} \right|^{\left( {n - 1} \right)} \left|
{S\left( {t,x} \right)} \right| dt}= 2\left[ {\left( {x - a}
\right)^{n + 1}  + \left( {\frac{{a + b}}{2} - x} \right)^{n + 1}
} \right]{\rm{B}}\left( {n,2} \right),
\end{align*}
where  ${\rm{B}}\left( {\cdot,\cdot} \right)$ is the Euler beta
function.  To argue the sharpness, we consider first when $1<p\le
\infty$, so that the equality in \eqref{eq2.1} holds when
\begin{align*}
f^{\left( n \right)} \left( t \right) = \left| {x - t}
\right|^{\left( {n - 1} \right)q - 1} \left| {S\left( {t,x}
\right)} \right|^{q - 1} {\mathop{\rm sgn}} \left\{ {\left( {x -
t} \right)^{n - 1} S\left( {t,x} \right)} \right\},
\end{align*}
thus the inequality \eqref{eq2.18} holds for $1<p\le \infty$. In
case that $p=1$, setting
\begin{align*}
g\left( {t,x} \right) = \left( {x - t} \right)^{n - 1} S\left(
{t,x} \right) \qquad \forall x\in \left[a,\textstyle{{a + b} \over
2}\right],
\end{align*}
let $t_0$ be the point that gives the supremum. If
$t_0=\frac{x+\left(n-1\right)a}{n}$, we take
\begin{align*}
f_{\varepsilon}^{\left( n \right)} \left( t \right) = \left\{
\begin{array}{l}
 \varepsilon ^{ - 1} ,\,\,\,\,\,\,t \in \left( {t_0  - \varepsilon ,t_0 } \right) \\
 0,\,\,\,\,\,\,\,\,\,\,{\rm{otherwise}} \\
 \end{array} \right..
\end{align*}
Since
\begin{align*}
\left| {\int_a^b {g\left( {t,x} \right)f^{\left( n \right)} \left(
t \right)dt} } \right| = \frac{1}{\varepsilon }\left| {\int_{t_0 -
\varepsilon }^{t_0 } {g\left( {t,x} \right)dt} } \right| &\le
\frac{1}{\varepsilon }\int_{t_0  - \varepsilon }^{t_0 } {\left|
{g\left( {t,x} \right)} \right|dt}
\\
&\le \mathop {\sup }\limits_{t_0  - \varepsilon  \le t \le t_0 }
\left| {g\left( {t,x} \right)} \right| \cdot \frac{1}{\varepsilon
}\int_{t_0  - \varepsilon }^{t_0 } {dt}
\\
&= \left| {g\left( {t_0 ,x} \right)} \right|,
\end{align*}
also, we have
\begin{align*}
\mathop {\lim }\limits_{\varepsilon  \to 0^ +  }
\frac{1}{\varepsilon }\int_{t_0  - \varepsilon }^{t_0 } {\left|
{g\left( {t,x} \right)} \right|dt}  = \left| {g\left( {t_0 ,x}
\right)} \right|=C\left( n,1,x \right)
\end{align*}
proving that $C\left( n,1,x \right)$ is the best possible.
\end{proof}

\begin{corollary}
Under the assumptions of Theorem \ref{thm4}.
\begin{enumerate}
\item If $k$ is even and $f^{\left( {k - 1} \right)} \left( {a}
\right)=f^{\left( {k - 1} \right)} \left( {b} \right)=0$, for all
$k=1,\cdots,n-1$. Then,
\begin{align}
\left|{ \frac{1}{n}\left( {\frac{{f\left( x \right) + f\left( {a +
b - x} \right)}}{2} } \right) - \frac{1}{{b - a}}\int_a^b {f\left(
y \right)dy} }\right| \le C\left( {n,p,x} \right)\left\|
{f^{\left( n \right)} } \right\|_p.\label{eq2.20}
\end{align}

\item If $f^{\left( {k - 1} \right)} \left( {a} \right)=f^{\left(
{k - 1} \right)} \left( {\frac{a+b}{2}} \right)=f^{\left( {k - 1}
\right)} \left( {b} \right)=0$, for all $k=1,\cdots,n-1$. Then the
inequality \eqref{eq2.20} holds.
\end{enumerate}

\end{corollary}

\begin{remark}
In Theorem \ref{thm6}, if one assumes that $f^{\left(n\right)}$ is
$n$-convex, $r$-convex, quasi-convex, $s$-convex, $P$-convex, or
$Q$-convex; we may obtain other new bounds involving convexity.
\end{remark}


\section{Bounds  via Chebyshev-Gr\"{u}ss type inequalities}

The celebrated \v{C}eby\v{s}ev functional
\begin{multline}
\label{eq3.1}\mathcal{C}\left( {h_1,h_1} \right)
\\
= \frac{1}{{d - c}}\int_c^d {h_1\left( t \right)h_2\left( t
\right)dt}  - \frac{1}{{d - c}}\int_c^d {h_1\left( t \right)dt}
\cdot \frac{1}{{d - c}}\int_c^d {h_2\left( t \right)dt}.
\end{multline}
has multiple applications in several subfields including Numerical
integrations, Probability Theory \& Statistics, Functional
Analysis, Operator Theory and others. For more detailed history
see \cite{MPF}.

The most famous bounds of the \v{C}eby\v{s}ev functional are
incorporated in the following theorem:
\begin{theorem}\label{thm7}
Let $f,g:[c,d] \to \mathbb{R}$ be two absolutely continuous
functions, then
\begin{align}
\left|{\mathcal{C}\left( {h_1,h_2} \right)} \right| \le\left\{
\begin{array}{l} \frac{{\left( {d - c} \right)^2 }}{{12}}\left\|
{h^{\prime}_1} \right\|_\infty  \left\| {h^{\prime}_2}
\right\|_\infty
,\,\,\,\,\,\,\,\,\,{\rm{if}}\,\,h^{\prime}_1,h^{\prime}_2 \in L_{\infty}\left(\left[c,d\right]\right),\,\,\,\,\,\,\,\,{\rm{proved \,\,in \,\,}}{\text{\cite{Cebysev}}}\\
 \\
\frac{1}{4}\left( {M_1 - m_1} \right)\left( {M_2 - m_2}
\right),\,\,\, {\rm{if}}\,\, m_1\le h_1 \le M_1,\,\,\,m_2\le h_2
\le
M_2, \,\,{\rm{proved \,\,in \,\,}}{\text{\cite{Gruss}}}\\
 \\
\frac{{\left( {d - c} \right)}}{{\pi ^2 }}\left\| {h^{\prime}_1}
\right\|_2 \left\| {h^{\prime}_2} \right\|_2
,\,\,\,\,\,\,\,\,\,\,\,\,\,\,\,\,{\rm{if}}\,\,h^{\prime}_1,h^{\prime}_2
\in
L_{2}\left(\left[c,d\right]\right),\,\,\,\,\,\,\,\,\,\,\,{\rm{proved
\,\,in\,\,
}}{\text{\cite{L}}}\\
\\
\frac{1}{8}\left( {d - c} \right)\left( {M - m} \right) \left\|
{h^{\prime}_2} \right\|_{\infty},\,\,\, {\rm{if}}\,\, m\le h_1 \le
M,\,h^{\prime}_2 \in L_{\infty}\left(\left[c,d\right]\right),
\,\,{\rm{proved \,\,in \,\,}}{\text{\cite{O}}}
\end{array} \right. \label{eq3.2}
\end{align}
The constants $\frac{1}{12}$, $\frac{1}{4}$, $\frac{1}{\pi^2}$ and
$\frac{1}{8}$ are the best possible.
\end{theorem}

In this section, we highlight the role of \v{C}eby\v{s}ev
functional in integral approximations by using the
\v{C}eby\v{s}ev--Gr\"{u}ss type inequalities \eqref{eq3.2}.

Setting $h_1\left(t\right)=\frac{1}{n!}f^{(n)}\left(t\right)$ and
$h_2\left(t\right)=\left(x-t\right)^{n-1}k\left(t,x\right)$,  we
have
\begin{align*}
\mathcal{C}\left( {h_1,h_2} \right) &= \frac{1}{n!\left(b-a\right)
} \int_a^b {\left( {x - t} \right)^{n - 1} k\left( {t,x}
\right)f^{\left( n \right)} \left( t \right)dt}
\\
&\qquad-  \frac{1}{n!} \cdot\frac{1}{b-a} \int_a^b {\left( {x - t}
\right)^{n - 1} k\left( {t,x} \right) dt} \cdot\frac{1}{b-a}
\int_a^b { f^{\left( n \right)} \left( t \right)dt}
\\
&= \frac{1}{n!\left(b-a\right) } \int_a^b {\left( {x - t}
\right)^{n - 1} k\left( {t,x} \right)f^{\left( n \right)} \left( t
\right)dt}
\\
&\qquad- \frac{{2 }}{{n!\left( {b - a} \right)}}\left[ {\left( {x
- a} \right)^{n + 1}  + \left( {\frac{{a + b}}{2} - x} \right)^{n
+ 1} } \right] {\rm{B}} \left( {n,2} \right)
\\
&\qquad\qquad\times \frac{{f^{\left( {n - 1} \right)} \left( b
\right) - f^{\left( {n - 1} \right)} \left( a \right)}}{{b - a}}
\end{align*}
which means
\begin{align*}
\mathcal{C}\left( {h_1,h_2} \right) &= \frac{1}{n}\left(
{\frac{{f\left( x \right) + f\left( {a + b - x} \right)}}{2} +
\sum\limits_{k = 1}^{n - 1} {G_k\left(x\right)} } \right) -
\frac{1}{{b - a}}\int_a^b {f\left( y \right)dy}
\\
&\qquad- \frac{{2 }}{{\left( {n + 1} \right)!n\left( {b - a}
\right)}}\left[ {\left( {x - a} \right)^{n + 1}  + \left(
{\frac{{a + b}}{2} - x} \right)^{n + 1} } \right]
\\
&\qquad\qquad\times \frac{{f^{\left( {n - 1} \right)} \left( b
\right) - f^{\left( {n - 1} \right)} \left( a \right)}}{{b - a}}
\\
&:=\mathcal{P}\left( {f;x,n} \right).
\end{align*}

\begin{theorem}
\label{thm8}Let $I$ be a real interval, $a,b \in I^{\circ}$
$(a<b)$. Let $f:I \to \mathbb{R}$ be $(n+1)$-times differentiable
on $I^{\circ}$ such that $f^{(n+1)}$ is absolutely continuous on
$I^{\circ}$ with $\left( {\cdot - t} \right)^{n - 1} k\left(
{t,\cdot} \right)f^{\left( n \right)} \left( t \right)$ is
integrable. Then, for all $n\ge2$ we have
\begin{multline}
\left|{\mathcal{P}\left( {f;x,n} \right)} \right|
\\
\le\left\{
\begin{array}{l}
\left( {b-a} \right)^{2}\left( {\frac{{n - 2}}{n}} \right)^{n - 2}
\frac{{n^2  - 2n + 2}}{{12n\cdot (n!)^2}} \left[ {\frac{{b -
a}}{4} + \left| {x - \frac{{3a + b}}{4}} \right|}
\right]^{n-1}\cdot \left\|{f^{\left( {n + 1} \right)}
}\right\|_{\infty} ,\,\,\,\,\,\,\,\,\,{\rm{if}}\,\,f^{(n+1)} \in
L_{\infty}\left(\left[a,b\right]\right)\\
 \\
\left( {\frac{{n - 2}}{n}} \right)^{n - 2} \frac{{n^2  - 2n +
2}}{{4n\cdot (n!)^2}}   \left( {2^{ - n - 2}  - 2^{ - 2n - 2} }
\right) \left( {b -  a} \right)^{n-2}\cdot  \left( {M -  m}
\right),\,\,\,
{\rm{if}}\,\, m\le f^{(n)} \le M, \\
 \\
 \frac{{b-a}}{{(n!)^2\pi^2}}
\sqrt {A\left( n \right)\left( {x - a} \right)^{2n - 1}  + B\left(
n \right)\left( {\frac{{a + b}}{2} - x} \right)^{2n - 1} }
 \cdot
\left\|{f^{\left( {n + 1} \right)} }\right\|_{2} ,\,\,\,\,
{\rm{if}}\,\,f^{(n+1)}
\in L_{2}\left(\left[a,b\right]\right),\\
\\
\left( {b-a} \right)\left( {\frac{{n - 2}}{n}} \right)^{n - 2}
\frac{{n^2  - 2n + 2}}{{8n\cdot (n!)^2}} \left[ {\frac{{b - a}}{4}
+ \left| {x - \frac{{3a + b}}{4}} \right|} \right]^{n-1}\cdot
\left( {M - m} \right),\,\,\,
{\rm{if}}\,\, m\le f^{(n)} \le M, \\
\\
\left( {\frac{{n - 2}}{n}} \right)^{n - 2} \frac{{n^2  - 2n +
2}}{{8n\cdot (n!)^2}} \left( {2^{ - n - 2}  - 2^{ - 2n - 2} }
\right) \left( {b -  a} \right)^{n}\cdot \left\|{f^{\left( {n + 1}
\right)} }\right\|_{\infty},\,\,\, {\rm{if}}\,\,f^{(n+1)} \in
L_{\infty}\left(\left[a,b\right]\right),
\end{array} \right.
\end{multline}
holds for all $x\in \left[a,\frac{a+b}{2}\right]$,  where
\begin{align*}
A\left( n \right) = \frac{{2\left( {n - 1} \right)^2 }}{{\left(
{2n - 1} \right)\left( {2n - 2} \right)\left( {2n - 3}
 \right)}}
\end{align*}
and
\begin{align*}
B\left( n \right) = \frac{{2^{2n - 3} \left( {2n - 1}
\right)\left( {2n - 2} \right) + 4n\left( {2n - 1} \right) + 2n^2
}}{{\left( {2n - 1} \right)\left( {2n - 2} \right)\left( {2n - 3}
\right)}}
\end{align*}
$\forall n\ge2$.
\end{theorem}

\begin{proof}
$\bullet$  If $f^{(n+1)}\in
L^{\infty}\left(\left[a,b\right]\right)$: Applying the first
inequality in \eqref{eq3.2}, it is not difficult to observe that
$\mathop {\sup}\limits_{a \le t \le b} \left\{
{\left|{h^{\prime}_1 \left( t \right)}\right|} \right\}
=\frac{1}{n!}\left\|{f^{\left( {n + 1} \right)}
}\right\|_{\infty}$ and
\begin{align*}
\mathop {\sup}\limits_{a \le t \le b} \left\{ {\left|{h^{\prime}_2
\left( t \right)}\right|} \right\} =\left( {\frac{{n - 2}}{n}}
\right)^{n - 2} \frac{{n^2  - 2n + 2}}{{n}} \left[ {\frac{{b -
a}}{4} + \left| {x - \frac{{3a + b}}{4}} \right|}
\right]^{n-1},\qquad \forall n\ge2.
\end{align*}
So that
\begin{align*}
\left|{\mathcal{P}\left( {f;x,n} \right)} \right| \le \left( {b-a}
\right)^2 \left( {\frac{{n - 2}}{n}} \right)^{n - 2} \frac{{n^2  -
2n + 2}}{{12n\cdot n!}} \left[ {\frac{{b - a}}{4} + \left| {x -
\frac{{3a + b}}{4}} \right|} \right]^{n-1}\cdot
\frac{1}{n!}\left\|{f^{\left( {n + 1} \right)} }\right\|_{\infty}.
\end{align*}

$\bullet$  If $m\le f^{(n)}\left(t\right) \le M$, for some
$m,M>0$: Applying the second inequality in \eqref{eq3.2}, we get
\begin{align*}
\left|{\mathcal{P}\left( {f;x,n} \right)} \right|\le
\frac{n^2-2n+2}{4n\cdot n!}\left( {\frac{{n - 2}}{n}} \right)^{n -
2} \left( {2^{ - n - 2}  - 2^{ - 2n - 2} } \right)\left( {b - a}
\right)^{n-2} \cdot \frac{1}{n!} \left( {M -  m} \right).
\end{align*}

$\bullet$  If $f^{(n+1)} \in L^{2}\left(\left[a,b\right]\right)$:
Applying the third inequality in \eqref{eq3.2}, we get
\begin{align*}
\left|{\mathcal{P}\left( {f;x,n} \right)} \right| \le
\frac{{\left( {b-a} \right)}}{{n!\pi^2}} \cdot   \sqrt {A\left( n
\right)\left( {x - a} \right)^{2n - 1}  + B\left( n \right)\left(
{\frac{{a + b}}{2} - x} \right)^{2n - 1} }
 \cdot
\frac{1}{n!}\left\|{f^{\left( {n + 1} \right)} }\right\|_{2}
\end{align*}
$\forall n\ge2$, where  $A\left( n \right)$ and $B\left( n
\right)$ are defined
above.\\

$\bullet$  If $m\le f^{(n)}\left(t\right) \le M$, for some
$m,M>0$: Applying the forth inequality in \eqref{eq3.2}, we get
\begin{align*}
\left|{\mathcal{P}\left( {f;x,n} \right)} \right| \le  \left(
{b-a} \right)\left( {\frac{{n - 2}}{n}} \right)^{n - 2} \frac{{n^2
- 2n + 2}}{8n \cdot n!} \left[ {\frac{{b - a}}{4} + \left| {x -
\frac{{3a + b}}{4}} \right|} \right]^{n-1}\cdot  \frac{1}{n!}
\left( {M - m} \right).
\end{align*}
By applying the forth inequality again the with dual assumptions,
i.e.,  $f^{(n+1)} \in L^{\infty}\left(\left[a,b\right]\right)$, we
have
\begin{align*}
\left|{\mathcal{P}\left( {f;x,n} \right)} \right|  &\le
\frac{n^2-2n+2}{8n\cdot n!}\left( {\frac{{n - 2}}{n}} \right)^{n -
2} \left( {2^{ - n - 2}  - 2^{ - 2n - 2} } \right) \left( {b - a}
\right)^{n}\cdot \frac{1}{n!}\left\|{f^{\left( {n + 1} \right)}
}\right\|_{\infty}.
\end{align*}
Hence the proof is completely established.
\end{proof}

\begin{corollary}
\label{cor2}Let assumptions of Theorem \ref{thm8} hold. If
moreover, $f^{\left( {n - 1} \right)}\left( {a} \right)=f^{\left(
{n - 1} \right)}\left( {b} \right)$ $(n \ge2)$, then the
inequality
\begin{multline}
\label{eq3.4}\left|{ \frac{1}{n}\left( {\frac{{f\left( x \right) +
f\left( {a + b - x} \right)}}{2} + \sum\limits_{k = 1}^{n - 1}
{G_k } } \right) - \frac{1}{{b - a}}\int_a^b {f\left( y \right)dy}
}\right|
\\
\le\left\{
\begin{array}{l}
\left( {b-a} \right)^{2}\left( {\frac{{n - 2}}{n}} \right)^{n - 2}
\frac{{n^2  - 2n + 2}}{{12n\cdot (n!)^2}} \left[ {\frac{{b -
a}}{4} + \left| {x - \frac{{3a + b}}{4}} \right|}
\right]^{n-1}\cdot \left\|{f^{\left( {n + 1} \right)}
}\right\|_{\infty} ,\,\,\,\,\,\,\,\,\,{\rm{if}}\,\,f^{(n+1)} \in
L_{\infty}\left(\left[a,b\right]\right)\\
 \\
\left( {\frac{{n - 2}}{n}} \right)^{n - 2} \frac{{n^2  - 2n +
2}}{{4n\cdot (n!)^2}}   \left( {2^{ - n - 2}  - 2^{ - 2n - 2} }
\right) \left( {b -  a} \right)^{n-2}\cdot  \left( {M -  m}
\right),\,\,\,
{\rm{if}}\,\, m\le f^{(n)} \le M, \\
 \\
 \frac{{b-a}}{{(n!)^2\pi^2}}
\sqrt {A\left( n \right)\left( {x - a} \right)^{2n - 1}  + B\left(
n \right)\left( {\frac{{a + b}}{2} - x} \right)^{2n - 1} }
 \cdot
\left\|{f^{\left( {n + 1} \right)} }\right\|_{2} ,\,\,\,\,
{\rm{if}}\,\,f^{(n+1)}
\in L_{2}\left(\left[a,b\right]\right),\\
\\
\left( {b-a} \right)\left( {\frac{{n - 2}}{n}} \right)^{n - 2}
\frac{{n^2  - 2n + 2}}{{8n\cdot (n!)^2}} \left[ {\frac{{b - a}}{4}
+ \left| {x - \frac{{3a + b}}{4}} \right|} \right]^{n-1}\cdot
\left( {M - m} \right),\,\,\,
{\rm{if}}\,\, m\le f^{(n)} \le M, \\
\\
\left( {\frac{{n - 2}}{n}} \right)^{n - 2} \frac{{n^2  - 2n +
2}}{{8n\cdot (n!)^2}} \left( {2^{ - n - 2}  - 2^{ - 2n - 2} }
\right) \left( {b -  a} \right)^{n}\cdot \left\|{f^{\left( {n + 1}
\right)} }\right\|_{\infty},\,\,\, {\rm{if}}\,\,f^{(n+1)} \in
L_{\infty}\left(\left[a,b\right]\right),
\end{array} \right.
\end{multline}
holds for all $x\in \left[a,\frac{a+b}{2}\right]$,  where
\begin{align*}
A\left( n \right) = \frac{{2\left( {n - 1} \right)^2 }}{{\left(
{2n - 1} \right)\left( {2n - 2} \right)\left( {2n - 3}
 \right)}}
\end{align*}
and
\begin{align*}
B\left( n \right) = \frac{{2^{2n - 3} \left( {2n - 1}
\right)\left( {2n - 2} \right) + 4n\left( {2n - 1} \right) + 2n^2
}}{{\left( {2n - 1} \right)\left( {2n - 2} \right)\left( {2n - 3}
\right)}}
\end{align*}
$\forall n \ge2$.
\end{corollary}

\begin{remark}
By setting
$h_1\left(t\right)=\frac{1}{n!}f^{(n)}\left(t\right)k\left(t,x\right)$
and $h_2\left(t\right)=\left(x-t\right)^{n-1}$, we obtain that
\begin{align}
\mathcal{C}\left( {h_1,h_2} \right) &= \frac{1}{n}\left(
{\frac{{f\left( x \right) + f\left( {a + b - x} \right)}}{2} +
\sum\limits_{k = 1}^{n - 1} {G_k\left(x\right) } } \right) -
\frac{1}{{b - a}}\int_a^b {f\left( y \right)dy}
\nonumber\\
&\qquad-  \frac{1}{n!}\cdot\frac{{\left( {x - a} \right)^n  -
\left( {x - b} \right)^n }}{n\left( {b - a} \right)}
\cdot\frac{f^{\left( n \right)} \left( x \right) + f^{\left( n
\right)} \left( {a + b - x} \right)}{2}
\nonumber\\
&:=\mathcal{Q}\left( {f;x,n} \right).\label{eq3.5}
\end{align}
Applying Theorem \ref{thm7}   Chebyshev type bounds for
$\mathcal{Q}\left( {f;x,n} \right)$ can be proved. We shall omit
the details.
\end{remark}

\section{Generalizations of the  results}

In this section, generalization of the identity \eqref{eq2.1} via
Harmonic sequence of polynomials through Fink's approach  is
considered. Generalizations of Guessab--Schmeisser formula
integral formula \eqref{eq1.9} which is of Euler--Maclaurin type
for symmetric values of real functions are established. Some norm
inequalities of these generalized formulae with some special cases
which are of great interests are also provided.

\begin{theorem}
\label{thm9}Let $I$ be a real interval, $a,b \in I^{\circ}$
$(a<b)$. Let $P_k$ be a harmonic sequence of polynomials and let
$f:I \to \mathbb{R}$ be such that $f^{(n)}$ is absolutely
continuous on $I$ for $n\ge1$ with $P_{n-1}\left( {t} \right)
S\left( {t,\cdot} \right)f^{\left( n \right)} \left( t \right)$ is
integrable. Then we have the representation
\begin{multline}
\frac{1}{n}\left[ {\frac{{f\left( x \right) + f\left( {a + b - x}
\right)}}{2} + \sum\limits_{k = 1}^{n - 1} { \left\{{ T_k
\left(x\right)+ \widetilde{F_k } \left(a,b\right) }\right\}} }
\right] - \frac{1}{{b - a}}\int_a^b {f\left( y \right)dy}
\\
= \frac{{\left( { - 1} \right)^{n - 1} }}{{\left( {b - a}
\right)n}}\int_a^b {P_{n - 1} \left( t \right)S\left( {t,x}
\right)f^{\left( n \right)} \left( t \right)dt},\label{eq4.1}
\end{multline}
for all $x\in\left[{a,\frac{a+b}{2}}\right]$, where
\begin{align}
\label{eq4.2}T_k  \left(x\right)= \frac{\left( { - 1}
\right)^k}{2} \left\{ {P_k \left( x \right)f^{\left( k \right)}
\left( x \right) + P_k \left( {a + b - x} \right)f^{\left( k
\right)} \left( {a + b - x} \right)} \right\}
\end{align}
$ \widetilde{F_k } \left(a,b\right)$ is given in \eqref{eq1.10}
and $S\left( {t,x} \right)$ as given in Theorem \ref{thm4}.
\end{theorem}

\begin{proof}
Fix $x\in \left[a,b\right]$. In the representation \eqref{eq1.9},
replace $b$ by $\frac{a+b}{2}$ we get
\begin{align}
\label{eq4.3}&\frac{1}{n}\left[ {f\left( x \right) +
\sum\limits_{k = 1}^{n - 1} {\left( { - 1} \right)^k P_k \left( x
\right)f^{\left( k \right)} \left( x \right)}  + 2\sum\limits_{k =
1}^{n - 1} {\widetilde{F_k }\left( {a,\frac{{a + b}}{2}} \right)}
} \right] - \frac{2}{{b - a}}\int_a^{\frac{{a + b}}{2}} {f\left( y
\right)dy}
\\
&= \frac{{2\left( { - 1} \right)^{n - 1} }}{{\left( {b - a}
\right)n}} \int_a^{\frac{{a + b}}{2}} {\left( {\int_y^x {P_{n - 1}
\left( t \right)f^{\left( n \right)} \left( t \right)dt} }
\right)dy},\nonumber
\end{align}
where $F_k$ i given in \eqref{eq1.10}. As a second step, in the
same formula \eqref{eq1.9} we replace every $x$ by $a+b-x$ and $a$
by $\frac{a+b}{2}$ for all $x\in \left[{\frac{a+b}{2},b}\right]$,
then $f$ has the representation
\begin{multline}
\frac{1}{n}\left[ {f\left( {a + b - x} \right) + \sum\limits_{k =
1}^{n - 1} {\left( { - 1} \right)^k P_k \left( {a + b - x}
\right)f^{\left( k \right)} \left( {a + b - x} \right)} } \right.
\\
\left.{+ 2\sum\limits_{k = 1}^{n - 1} {\widetilde{F_k }\left(
{\frac{{a + b}}{2},b} \right)} } \right] - \frac{2}{{b -
a}}\int_{\frac{{a + b}}{2}}^b {f\left( y \right)dy}
\\
= \frac{{2\left( { - 1} \right)^{n - 1} }}{{\left( {b - a}
\right)n}} \int_{\frac{{a + b}}{2}}^b {\left( {\int_y^{a + b - x}
{P_{n - 1} \left( t \right)f^{\left( n \right)} \left( t
\right)dt} } \right)dy},\label{eq4.4}
\end{multline}
Multiplying \eqref{eq4.3} and \eqref{eq4.4} by $\frac{1}{2}$ and
then adding the corresponding equations, we get
\begin{align}
&\frac{1}{n}\left[ {\frac{{f\left( x \right) + f\left( {a + b - x}
\right)}}{2} }\right. \label{eq4.5}
\\
&\qquad+ \frac{1}{2}\sum\limits_{k = 1}^{n - 1} {\left( { - 1}
\right)^k \left\{ {P_k \left( x \right)f^{\left( k \right)} \left(
x \right) + P_k \left( {a + b - x} \right)f^{\left( k \right)}
\left( {a + b - x} \right)} \right\}}
\nonumber\\
&\qquad \left. { + \sum\limits_{k = 1}^{n - 1} {  \left\{
{\widetilde{F_k }\left( {a,\frac{{a + b}}{2}} \right) +
\widetilde{F_k }\left( {\frac{{a + b}}{2},b} \right)} \right\}} }
\right]- \frac{1}{{b - a}}\int_a^b {f\left( y \right)dy} \nonumber
\\
&= \frac{{\left( { - 1} \right)^{n - 1} }}{{\left( {b - a}
\right)n}} \left[{ \int_a^{\frac{{a + b}}{2}} {\left( {\int_y^x
{P_{n - 1} \left( t \right)f^{\left( n \right)} \left( t
\right)dt} } \right)dy}  + \int_{\frac{{a + b}}{2}}^b {\left(
{\int_y^{a + b - x} {P_{n - 1} \left( t \right)f^{\left( n
\right)} \left( t \right)dt} } \right)dy} }\right].\nonumber
\end{align}
But since
\begin{align*}
&\widetilde{F_k }\left( {a,\frac{{a + b}}{2}} \right) +
\widetilde{F_k }\left( {\frac{{a + b}}{2},b} \right)
\nonumber\\
&= \frac{{\left( { - 1} \right)^k \left( {n - k} \right)}}{{b -
a}}\left[ {P_k \left( a \right)f^{\left( {k - 1} \right)} \left( a
\right) - P_k \left( {\frac{{a + b}}{2}} \right)f^{\left( {k - 1}
\right)} \left( {\frac{{a + b}}{2}} \right)} \right]
\nonumber\\
&\qquad+ \left[ {  P_k \left( {\frac{{a + b}}{2}} \right)f^{\left(
{k - 1} \right)} \left( {\frac{{a + b}}{2}} \right) -   P_k \left(
b\right)f^{\left( {k - 1} \right)} \left( b \right)} \right]
\nonumber\\
&= \widetilde{F_k }\left( {a,b} \right),
\end{align*}
then \eqref{eq4.5} becomes
\begin{align}
&\frac{1}{n}\left[ {\frac{{f\left( x \right) + f\left( {a + b - x}
\right)}}{2} + \sum\limits_{k = 1}^{n - 1} { \left\{{ T_k
\left(x\right)+ \widetilde{F_k } \left(a,b\right) }\right\}} }
\right] - \frac{1}{{b - a}}\int_a^b {f\left( y
\right)dy}\label{eq4.6}
\\
&= \frac{{\left( { - 1} \right)^{n - 1} }}{{\left( {b - a}
\right)n}} \left[{ \int_a^{\frac{{a + b}}{2}} {\left( {\int_y^x
{P_{n - 1} \left( t \right)f^{\left( n \right)} \left( t
\right)dt} } \right)dy}  + \int_{\frac{{a + b}}{2}}^b {\left(
{\int_y^{a + b - x} {P_{n - 1} \left( t \right)f^{\left( n
\right)} \left( t \right)dt} } \right)dy} }\right].\nonumber
\end{align}
Also,  the right hand-side can be simplified as shown in
\eqref{eq2.14}--\eqref{eq2.16}, i.e., we have
\begin{multline*}
 \int_a^{\frac{{a + b}}{2}} {\left( {\int_y^x {P_{n - 1}
\left( t \right)f^{\left( n \right)} \left( t \right)dt} }
\right)dy}  + \int_{\frac{{a + b}}{2}}^b {\left( {\int_y^{a + b -
x} {P_{n - 1} \left( t \right)f^{\left( n \right)} \left( t
\right)dt} } \right)dy}
\\
=\int_a^b {P_{n - 1} \left( t \right)S\left( {t,x}
\right)f^{\left( n \right)} \left( t \right)dt},
\end{multline*}
where
\begin{align*}
S\left( {t,x} \right) =  \left\{
\begin{array}{l}
 t - a,\,\,\,\,\,\,\,\,\,\,\,\,\,\,\,\,\,\, t \in \left[ {a,x} \right] \\
 t - \frac{{a + b}}{2},\,\,\,\,\,\,\,\,\,\,\,t \in \left( {x,a + b - x} \right) \\
 t - b,\,\,\,\,\,\,\,\,\,\,\,\,\,\,\,\,\,\, t \in \left[ {a + b - x,b} \right] \\
 \end{array} \right..
\end{align*}
for all $x\in\left[ {a,{\textstyle{{a + b} \over 2}}} \right]$,
which gives the desired  representation in \eqref{eq4.1}.
\end{proof}

\begin{corollary}
\label{cor3}Under the assumptions of Theorem \ref{thm9}, we have
\begin{multline}
\frac{1}{n}\left[ {\frac{{f\left( x \right) + f\left( {a + b - x}
\right)}}{2} + \sum\limits_{k = 1}^{n - 1} { \left\{ {W\left(
{x,y} \right) + F_k \left( y \right)} \right\} }  } \right] -
\frac{1}{{b - a}}\int_a^b {f\left( y \right)dy}
\\
= \frac{{1 }}{{\left( {b - a} \right)n!}}\int_a^b
{\left(y-t\right)^{n-1}S\left( {t,x} \right)f^{\left( n \right)}
\left( t \right)dt} \label{eq4.7}
\end{multline}
where
\begin{align*}
W\left( {x,y} \right)=\frac{\left( { - 1} \right)^k }{2} \left\{ {
\left(x-y\right)^k f^{\left( k \right)} \left( x \right) +
\left(a+b-x-y\right)^k f^{\left( k \right)} \left( {a + b - x}
\right)} \right\}
\end{align*}
for all $x\in\left[{a,\frac{a+b}{2}}\right]$ and all $y\in
\left[a,b\right]$, where $S\left( {t,x} \right)$ is given in
Theorem \ref{thm4} and $F_k\left(y\right)$ is given in
\eqref{eq1.4}
\end{corollary}

\begin{proof}
In \eqref{eq4.1}, choose
$P_k\left(t\right)=\frac{\left(t-y\right)^k}{k!}$, we get the
desired representation \eqref{eq4.7}.
\end{proof}

A Guessab--Schmeisser like expansion (see Theorem \ref{thm3}) may
be deduced as follows:
\begin{corollary}
\label{cor4} Under the assumptions of Theorem \ref{thm9}.
Additionally if $P_k \left( t \right) = \left( { - 1} \right)^k
P_k \left( {a + b - t} \right), \forall t\in \left[ {a,b}
\right]$, then
\begin{multline}
\frac{1}{n}\left[ {\frac{{f\left( x \right) + f\left( {a + b - x}
\right)}}{2} + \sum\limits_{k = 1}^{n - 1} {\left\{{
\widetilde{\widetilde{T_k }}\left( x \right)+\widetilde{F_k }
\left(a,b\right)}\right\}} } \right] - \frac{1}{{b - a}}\int_a^b
{f\left( y \right)dy}
\\
= \frac{{\left( { - 1} \right)^{n - 1} }}{{\left( {b - a}
\right)n}}\int_a^b {P_{n - 1} \left( t \right)S\left( {t,x}
\right)f^{\left( n \right)} \left( t \right)dt} \label{eq4.8}
\end{multline}
where
\begin{align*}
\widetilde{\widetilde{T_k }}\left( x \right)=\frac{\left( { - 1}
\right)^k }{2} P_k \left( x \right)\left[ {f^{\left( k \right)}
\left( x \right) + \left( { - 1} \right)^kf^{\left( k \right)}
\left( {a + b - x} \right)} \right],
\end{align*}
for all $x\in \left[{a,\frac{a+b}{2}}\right]$.
\end{corollary}

\begin{proof}
Since $P_k \left( t \right) = \left( { - 1} \right)^k P_k \left(
{a + b - t} \right), \forall t\in \left[ {a,b} \right]$,
substituting in \eqref{eq4.1} we get the required result.
\end{proof}
It is conveient to remark here, from \eqref{eq4.8}  we can deduce
\eqref{eq4.1} by substituting
$P_k\left(t\right)=\frac{\left(t-x\right)^k}{k!}$  in
\eqref{eq4.8}, so that we get
\begin{multline}
\frac{1}{n}\left[ {\frac{{f\left( x \right) + f\left( {a + b - x}
\right)}}{2} + \sum\limits_{k = 1}^{n - 1} { \widetilde{F_k }
\left(a,b\right)} } \right] - \frac{1}{{b - a}}\int_a^b {f\left( y
\right)dy}
\\
= \frac{1}{{n!\left({b-a}\right) }}\int_a^b {\left( {x - t}
\right)^{n - 1} S\left( {t,x} \right)f^{\left( n \right)} \left( t
\right)dt}. \label{eq4.9}
\end{multline}
Clearly, the desired deduction is finished once we observe that $
\widetilde{F_k } \left(a,b\right)=G_k \left(x\right)$. Since $P_k
\left( b \right) = \left( { - 1} \right)^k P_k \left( {a}
\right)$, then
\begin{align*}
\widetilde{F_k }\left( {a,b} \right) &= \frac{{\left( { - 1}
\right)^k \left( {n - k} \right)}}{{b - a}}P_k \left(a
\right)\left[ {f^{\left( {k - 1} \right)} \left( a \right) -\left(
{ - 1} \right)^k f^{\left( {k - 1} \right)} \left( b \right)}
\right]
\\
&=\frac{{\left( { - 1} \right)^k \left( {n - k} \right)}}{{b -
a}}\frac{\left(a-x\right)^k}{k!}\left[ {f^{\left( {k - 1} \right)}
\left( a \right) - \left( { - 1} \right)^k f^{\left( {k - 1}
\right)} \left( b \right)} \right]
\\
&=\frac{{ \left( {n - k} \right)}}{{b -
a}}\frac{\left(x-a\right)^k}{k!}\left[ {f^{\left( {k - 1} \right)}
\left( a \right) + \left( { - 1} \right)^{k+1} f^{\left( {k - 1}
\right)} \left( b \right)} \right].
\end{align*}
Also, we note that
\begin{align*}
P_k\left(\frac{a+b}{2}\right)=\frac{\left(\frac{a+b}{2}-x\right)^k}{k!}
&=\left({-1}\right)^k\frac{\left(x-\frac{a+b}{2}\right)^k}{k!}
\\
&=\left({-1}\right)^k P_k\left(\frac{a+b}{2}\right)=
\left({-1}\right)^kP_k\left(a+b-\frac{a+b}{2}\right),
\end{align*}
this gives that
\begin{align*}
0=P_k\left(\frac{a+b}{2}\right)-\left({-1}\right)^k
P_k\left(\frac{a+b}{2}\right)&=\left({1+\left({-1}\right)^{k+1}}\right)
P_k\left(\frac{a+b}{2}\right).
\end{align*}
By our choice of $P_k$; we have $P_k\left(\frac{a+b}{2}\right)=
\frac{\left(\frac{a+b}{2}-x\right)^k}{k!}$, for all $x\in \left[
{a,{\textstyle{{a + b} \over 2}}} \right]$, therefore we can write
\begin{align*}
\widetilde{F_k }\left( {a,b} \right)+0 &=\widetilde{F_k }\left(
{a,b} \right)+\left({1+\left({-1}\right)^{k+1}}\right)
P_k\left(\frac{a+b}{2}\right)
\\
&= \frac{{ \left( {n - k} \right)}}{{\left(b - a\right)k!}}
\left[{\left(x-a\right)^k\left( {f^{\left( {k - 1} \right)} \left(
a \right) + \left( { - 1} \right)^{k+1} f^{\left( {k - 1} \right)}
\left( b \right)} \right) }\right.
\\
&\qquad\qquad\left.{+ \left({1+\left({-1}\right)^{k+1}}\right)
\left(\frac{a+b}{2}-x\right)^k}\right]
\\
&= G_k \left(x\right), {\rm{which \,\, is\,\, given \,\,in
\,\,\eqref{eq2.2}}}.
\end{align*}
Hence, the representation \eqref{eq4.8} reduces to \eqref{eq4.1}.


\begin{theorem}
\label{thm10} Under the assumptions of Theorem \ref{thm9}, we have
\begin{multline}
\left|{ \frac{1}{n}\left[ {\frac{{f\left( x \right) + f\left( {a +
b - x} \right)}}{2} + \sum\limits_{k = 1}^{n - 1} { \left\{{ T_k
\left(x\right)+ \widetilde{F_k } \left(a,b\right) }\right\}} }
\right] - \frac{1}{{b - a}}\int_a^b {f\left( y \right)dy}} \right|
\\
\le N\left(f;x,a,b\right)\cdot \left\| {f^{\left( n \right)} }
\right\|_p,\label{eq4.10}
\end{multline}
$\forall p\in \left[1,\infty\right]$ and all
$x\in\left[{a,\frac{a+b}{2}}\right]$, where
\begin{align}
\label{eq4.11}N\left(f;x,a,b\right):=\frac{1}{{n
\left({b-a}\right) }} \left\{
\begin{array}{l}
 \mathop {\sup }\limits_{a \le t \le b} \left\{ {\left| {P_{n - 1}
\left( t \right)} \right| \left| {S\left( {t,x} \right)} \right|} \right\},\,\,\,\,\,\,\,\,\,\,\,\,\,\,\,\,\,\,\,\,\,\,\,\,\,\,p = 1 \\
 \\
  \left( {\int_a^b {\left| {P_{n - 1}
\left( t \right)} \right|^{q} \left| {S\left( {t,x} \right)} \right|^q dt} } \right)^{1/q} ,\,\,\,\,\,\,\,\,\,\,1 < p < \infty  \\
 \\
   \int_a^b {\left| {P_{n - 1}
\left( t \right)} \right| \left| {S\left( {t,x} \right)} \right|dt} ,\,\,\,\,\,\,\,\,\,\,\,\,\,\,\,\,\,\,\,\,\,\,\,\,\,\,\,\,\,\,\,\,\,p = \infty  \\
 \end{array} \right.,
\end{align}
where $T_k \left(x\right)$ is given in \eqref{eq4.2} and $
\widetilde{F_k } \left(a,b\right)$ is given in \eqref{eq1.10}.
\end{theorem}

\begin{proof}
Utilizing the triangle integral inequality on the identity
\eqref{eq4.1} and employing some known norm inequalities we get
\begin{align*}
&\left|{ \frac{1}{n}\left[ {\frac{{f\left( x \right) + f\left( {a
+ b - x} \right)}}{2} + \sum\limits_{k = 1}^{n - 1} { \left\{{ T_k
\left(x\right)+ \widetilde{F_k } \left(a,b\right) }\right\}} }
\right] - \frac{1}{{b - a}}\int_a^b {f\left( y \right)dy}} \right|
\\
&\le \frac{1}{{n \left({b-a}\right) }}\int_a^b {\left| {P_{n - 1}
\left( t \right)} \right| \left|{S\left( {t,x} \right)}\right|
\left|{f^{\left( n \right)} \left( t \right)}\right|dt}
\\
&\le  \frac{1}{{n \left({b-a}\right) }} \left\{ \begin{array}{l}
 \left\| {f^{\left( n \right)} } \right\|_1 \mathop {\sup }\limits_{a \le t \le b} \left\{ {\left| {P_{n - 1}
\left( t \right)} \right| \left| {S\left( {t,x} \right)} \right|} \right\},\,\,\,\,\,\,\,\,\,\,\,\,\,\,\,\,\,\,\,\,\,\,\,\,\,\,p = 1 \\
 \\
 \left\| {f^{\left( n \right)} } \right\|_p \left( {\int_a^b {\left| {P_{n - 1}
\left( t \right)} \right|^{q} \left| {S\left( {t,x} \right)} \right|^q dt} } \right)^{1/q} ,\,\,\,\,\,\,\,\,\,\,1 < p < \infty  \\
 \\
 \left\| {f^{\left( n \right)} } \right\|_\infty  \int_a^b {\left| {P_{n - 1}
\left( t \right)} \right| \left| {S\left( {t,x} \right)} \right|dt} ,\,\,\,\,\,\,\,\,\,\,\,\,\,\,\,\,\,\,\,\,\,\,\,\,\,\,\,\,\,\,\,\,\,p = \infty  \\
 \end{array} \right..\\
\\
&=N\left(f;x,a,b\right)  \left\| {f^{\left( n \right)} }
\right\|_p, \qquad \forall p, \,\,1\le p \le \infty
\end{align*}
where $N\left(f;x,a,b\right)$ is defined in \eqref{eq4.11}, and
this completes the proof.
\end{proof}

\begin{corollary}
\label{cor5}Under the assumptions of Theorem \ref{thm10}, we have
\begin{multline}
\left|{ \frac{1}{n}\left[ {\frac{{f\left( x \right) + f\left( {a +
b - x} \right)}}{2} + \sum\limits_{k = 1}^{n - 1} { \left\{{ T_k
\left(x\right)+ \widetilde{F_k } \left(a,b\right) }\right\}} }
\right] - \frac{1}{{b - a}}\int_a^b {f\left( y \right)dy}} \right|
\\
\le \frac{1}{n}\left[ {\frac{1}{4} + \frac{\left| {x - \frac{{3a +
b}}{4}} \right|}{b-a}} \right]\cdot\left\| {P_{n-1}}
\right\|_q\cdot \left\| {f^{\left( n \right)} }
\right\|_p,\label{eq4.12}
\end{multline}
$\forall p,q \ge 1$ with $\frac{1}{p}+\frac{1}{q}=1$ and all
$x\in\left[{a,\frac{a+b}{2}}\right]$, where
\begin{align*}
\left\| {P_{n-1}} \right\|_q=\left\{
\begin{array}{l}
 \mathop {\sup }\limits_{a \le t \le b} \left\{ {\left| {P_{n - 1}
\left( t \right)} \right|} \right\},\,\,\,\,\,\,\,\,\,\,\,\,\,\,\,\,\,\,\,\,\,\,\,\,\,\,q = \infty \\
 \\
  \left( {\int_a^b {\left| {P_{n - 1}
\left( t \right)} \right|^{q}dt} } \right)^{1/q} ,\,\,\,\,\,\,\,\,\,\,1 < q < \infty  \\
 \\
   \int_a^b {\left| {P_{n - 1}
\left( t \right)} \right| dt} ,\,\,\,\,\,\,\,\,\,\,\,\,\,\,\,\,\,\,\,\,\,\,\,\,\,\,\,\,\,\,\,\,\,q = 1  \\
 \end{array} \right.
\end{align*}
\end{corollary}

\begin{proof}
In \eqref{eq4.10}, it is easy to verify that
\begin{align*}
N\left(f;x,a,b\right)&\le \frac{1}{n \left(b-a\right)}\mathop
{\sup }\limits_{a \le t \le b} \left\{ { \left| {S\left( {t,x}
\right)} \right|} \right\}\cdot \left\| {P_{n-1}} \right\|_q
\\
&= \frac{1}{n}\left[ {\frac{1}{4} + \frac{\left| {x - \frac{{3a +
b}}{4}} \right|}{b-a}} \right]\cdot\left\| {P_{n-1}} \right\|_q,
\end{align*}
$\forall q\in \left[1,\infty\right]$ and all
$x\in\left[{a,\frac{a+b}{2}}\right]$.
\end{proof}

\begin{remark}
In Theorem \ref{thm10}, if one assumes that $f^{\left(n\right)}$
is convex, $r$-convex, quasi-convex, $s$-convex, $P$-convex, or
$Q$-convex; we may obtain other new bounds involving convexity.
\end{remark}

\begin{remark}
\label{rem4}Bounds for the generalized formula \eqref{eq4.1} via
Chebyshev-Gr\"{u}ss type inequalities can be done by setting
$h_1\left(t\right)=\frac{(-1)^{n-1}}{n}f^{(n)}\left(t\right)$ and
$h_2\left(t\right)=P_{n-1}\left(t\right)S\left(t,x\right)$,
therefore we have
\begin{align*}
&\mathcal{C}\left( {h_1,h_2} \right)
\\
&= \frac{(-1)^{n-1}}{n\left(b-a\right) } \int_a^b
{P_{n-1}\left(t\right) S\left( {t,x} \right)f^{\left( n \right)}
\left( t \right)dt}
\\
&\qquad-   \frac{1}{b-a} \int_a^b {P_{n-1}\left(t\right) S\left(
{t,x} \right) dt} \times\frac{(-1)^{n-1}}{n\left(b-a\right) }
\int_a^b { f^{\left( n \right)} \left( t \right)dt}
\\
&= \frac{1}{n\left(b-a\right) } \int_a^b {P_{n-1}\left(t\right)
S\left( {t,x} \right)f^{\left( n \right)} \left( t \right)dt}
\\
&\qquad-   \frac{1}{b-a} \int_a^b {P^{\prime}_{n}\left(t\right)
S\left( {t,x} \right) dt} \times\frac{(-1)^{n-1}}{n}\cdot
\frac{{f^{\left( {n - 1} \right)} \left( b \right) - f^{\left( {n
- 1} \right)} \left( a \right)}}{{b - a}}
\end{align*}
\begin{align*}
&= \frac{(-1)^{n-1}}{n\left(b-a\right) } \int_a^b
{P_{n-1}\left(t\right) S\left( {t,x} \right)f^{\left( n \right)}
\left( t \right)dt}
\\
&\qquad-   \left[ {\frac{{P_n \left( x \right) + P_n \left( {a + b
- x} \right)}}{2} - \frac{{P_{n + 1} \left( b \right) - P_{n + 1}
\left( a \right)}}{{b - a}}} \right]
\\
&\qquad\qquad\times\frac{(-1)^{n-1}}{n}\cdot \frac{{f^{\left( {n -
1} \right)} \left( b \right) - f^{\left( {n - 1} \right)} \left( a
\right)}}{{b - a}}
\\
&=\mathcal{L}\left({f,P_n,x}\right).
\end{align*}
We left the representations to the reader.
\end{remark}

\begin{theorem}
\label{thm11}Let $I$ be a real interval, $a,b \in I^{\circ}$
$(a<b)$. Let $f:I \to \mathbb{R}$ be $(n+1)$-times differentiable
on $I^{\circ}$ such that $f^{(n+1)}$ is absolutely continuous on
$I^{\circ}$ with $\left( {\cdot - t} \right)^{n - 1} k\left(
{t,\cdot} \right)f^{\left( n \right)} \left( t \right)$ is
integrable. Then, for all $n\ge2$ we have
\begin{multline}
\left|{\mathcal{L}\left({f,P_n,x}\right)} \right|
\\
\le\left\{
\begin{array}{l}
\frac{\left( {b - a} \right)^2}{12n}\left\| {P_{n - 1}  + P_{n -
2} S\left( { \cdot ,x} \right)} \right\|_\infty \cdot
\left\|{f^{\left( {n + 1} \right)} }\right\|_{\infty}
,\,\,\,\,{\rm{if}}\,\,f^{(n+1)} \in
L_{\infty}\left(\left[a,b\right]\right)\\
 \\
\frac{1}{4n}\left( {M_1  - m_1 } \right)\left( {M_2  - m_2 }
\right),\,\,\,\,\,\,\,\,\,\,\,\,\,\,\,\,\,\,\,\,\,\,\,\,\,\,\,\,\,\,\,\,\,\,\,\,\,\,\,
\qquad
{\rm{if}}\,\, m_1\le f^{(n)} \le M_1, \\
 \\
\frac{b-a}{\pi^2n}D\left( {n,x} \right)  \cdot \left\|{f^{\left(
{n + 1} \right)} }\right\|_{2}
,\qquad\qquad\qquad\qquad\,\,\,\,\,\,\,  {\rm{if}}\,\,f^{(n+1)}
\in L_{2}\left(\left[a,b\right]\right),\\
\\
\frac{b-a}{8n}\left\| {P_{n - 1}  + P_{n - 2} S\left( { \cdot ,x}
\right)} \right\|_\infty \cdot    \left( {M_1 - m_1}
\right),\,\,\,\,\,\,\,\,\,\,\,{\rm{if}}\,\, m_1\le f^{(n)} \le M_1, \\
\\
\frac{b-a}{8n}\left( {M_2  - m_2 } \right)\cdot \left\|{f^{\left(
{n + 1} \right)}
}\right\|_{\infty},\,\,\,\,\,\,\,\,\qquad\qquad\qquad
{\rm{if}}\,\,f^{(n+1)} \in
L_{\infty}\left(\left[a,b\right]\right),
\end{array} \right.\label{eq4.13}
\end{multline}
holds for all $x\in \left[a,\frac{a+b}{2}\right]$,  where
\begin{align*}
M_2 : = \mathop {\max }\limits_{a \le t \le b} \left\{ {P_{n - 1}
\left( t \right)S\left( {t,x} \right)} \right\},\,\, m_2 : =
\mathop {\min }\limits_{a \le t \le b} \left\{ {P_{n - 1} \left( t
\right)S\left( {t,x} \right)} \right\}
\end{align*}
and
\begin{align*}
D\left( {n,x} \right) = \left( {\int_a^b {\left| {P_{n - 1} \left(
t \right) + P_{n - 2} \left( t \right)S\left( {t,x} \right)}
\right|^2 dt} } \right)^{1/2} \qquad \forall n\ge2.
\end{align*}
\end{theorem}

\begin{proof}
The proof of the result follows directly by applying Theorem
\ref{eq3.2} to the functions
$h_1\left(t\right)=\frac{(-1)^{n-1}}{n}f^{(n)}\left(t\right)$ and
$h_2\left(t\right)=P_{n-1}\left(t\right)S\left(t,x\right)$ as
shown previously in Remark \ref{rem4} and the rest of the proof
done using Theorem \ref{thm7}.
\end{proof}

\begin{corollary}
\label{cor6}Let assumptions of Theorem \ref{thm11} hold. If
moreover, $f^{\left( {n - 1} \right)}\left( {a} \right)=f^{\left(
{n - 1} \right)}\left( {b} \right)$ $(n \ge2)$, then the
inequality
\begin{multline}
\left|{\frac{1}{n}\left[ {\frac{{f\left( x \right) + f\left( {a +
b - x} \right)}}{2} + \sum\limits_{k = 1}^{n - 1} { \left\{{ T_k
\left(x\right)+ \widetilde{F_k } \left(a,b\right) }\right\}} }
\right] - \frac{1}{{b - a}}\int_a^b {f\left( y \right)dy} }
\right|
\\
\le\left\{
\begin{array}{l}
\frac{\left( {b - a} \right)^2}{12n}\left\| {P_{n - 1}  + P_{n -
2} S\left( { \cdot ,x} \right)} \right\|_\infty \cdot
\left\|{f^{\left( {n + 1} \right)} }\right\|_{\infty}
,\,\,\,\,{\rm{if}}\,\,f^{(n+1)} \in
L_{\infty}\left(\left[a,b\right]\right)\\
 \\
\frac{1}{4n}\left( {M_1  - m_1 } \right)\left( {M_2  - m_2 }
\right),
\,\,\,\,\,\,\,\,\,\,\,\,\,\,\,\,\,\,\,\,\,\,\,\,\,\,\,\,\,\,\,\,\,\,\,\,\,\,\qquad
{\rm{if}}\,\, m_1\le f^{(n)} \le M_1, \\
 \\
\frac{b-a}{\pi^2n}D\left( {n,x} \right)  \cdot \left\|{f^{\left(
{n + 1} \right)} }\right\|_{2}
,\qquad\qquad\qquad\qquad\,\,\,\,\,\, {\rm{if}}\,\,f^{(n+1)}
\in L_{2}\left(\left[a,b\right]\right),\\
\\
\frac{b-a}{8n}\left\| {P_{n - 1}  + P_{n - 2} S\left( { \cdot ,x}
\right)} \right\|_\infty \cdot    \left( {M_1 - m_1}
\right),\,\,\,\,\,\,\,\,\,\,{\rm{if}}\,\, m_1\le f^{(n)} \le M_1, \\
\\
\frac{b-a}{8n}\left( {M_2  - m_2 } \right)\cdot \left\|{f^{\left(
{n + 1} \right)}
}\right\|_{\infty},\,\,\,\,\,\,\,\qquad\qquad\qquad
{\rm{if}}\,\,f^{(n+1)} \in
L_{\infty}\left(\left[a,b\right]\right),
\end{array} \right.\label{eq4.14}
\end{multline}
holds for all $x\in \left[a,\frac{a+b}{2}\right]$.
\end{corollary}

\section{Quadrature  rules and error bounds}\label{sec5}

\subsection{Representations of  Quadratures} In viewing \eqref{eq2.1}, the integral $\int_a^b
{f\left( y \right)dy}$ can be expressed by the general quadrature
rule:
\begin{align}
\int_a^b {f\left( y \right)dy} =
\mathcal{Q}_n\left({f,x}\right)+\mathcal{E}_n\left({f,x}\right)\label{eq5.1}
\end{align}
where
\begin{align}
\mathcal{Q}_n\left({f,x}\right):=\frac{b-a}{n}\left(
{\frac{{f\left( x \right) + f\left( {a + b - x} \right)}}{2} +
\sum\limits_{k = 1}^{n - 1} {G_k\left(x\right) } }
\right),\label{eq5.2}
\end{align}
and
\begin{align}
\mathcal{E}_n\left({f,x}\right):=- \frac{1}{{n!}}\int_a^b {\left(
{x - t} \right)^{n - 1} k\left( {t,x} \right)f^{\left( n \right)}
\left( t \right)dt}.\label{eq5.3}
\end{align}
for all $a \le x \le \frac{a+b}{2}$.

In particular cases, we have:

$\bullet$ If $x=a$, then
\begin{align}
\int_a^b {f\left( y \right)dy} =
\mathcal{Q}_n\left({f,a}\right)+\mathcal{E}_n\left({f,a}\right)\label{eq5.4}
\end{align}
such that
\begin{align*}
\mathcal{Q}_n\left({f,a}\right):=\frac{b-a}{n}\left(
{\frac{{f\left( a \right) + f\left( {b} \right)}}{2} +
\sum\limits_{k = 1}^{n - 1} {G^a_k } } \right),
\end{align*}
and
\begin{align*}
\mathcal{E}_n\left({f,a}\right):=- \frac{1}{{n!}}\int_a^b {\left(
{a - t} \right)^{n - 1} k\left( {t,a} \right)f^{\left( n \right)}
\left( t \right)dt},
\end{align*}
where,
\begin{align*}
G^a_k  = \frac{{\left( {n - k} \right)}}{{k!}} \cdot   \left( {1 +
\left( { - 1} \right)^{k + 1} } \right)\left( {\frac{{b-a}}{2} }
\right)^k f^{\left( {k - 1} \right)} \left( {\frac{{a + b}}{2}}
\right),
\end{align*}
and $k\left( {t,a} \right) =
 t - \frac{{a + b}}{2},$ for all $t \in \left( {a,b}
 \right)$.

$\bullet$ If $x=\frac{3a+b}{4}$, then
\begin{align}
\int_a^b {f\left( y \right)dy} =
\mathcal{Q}_n\left({f,\frac{3a+b}{4}}\right)+\mathcal{E}_n\left({f,\frac{3a+b}{4}}\right)\label{eq5.5}
\end{align}
such that
\begin{align*}
\mathcal{Q}_n\left({f,\frac{3a+b}{4}}\right):=\frac{b-a}{n}\left(
{\frac{{f\left( \frac{3a+b}{4} \right) + f\left( {\frac{a+3b}{4}}
\right)}}{2} + \sum\limits_{k = 1}^{n - 1} {G^{\frac{3a+b}{4}}_k }
} \right),
\end{align*}
and
\begin{align*}
\mathcal{E}_n\left({f,\frac{3a+b}{4}}\right):=-
\frac{1}{{n!}}\int_a^b {\left( {\frac{3a+b}{4} - t} \right)^{n -
1} k\left( {t,\frac{3a+b}{4}} \right)f^{\left( n \right)} \left( t
\right)dt},
\end{align*}
where,
\begin{multline*}
G^{\frac{3a+b}{4}}_k  = \frac{{\left( {n - k} \right)}}{{k!}}
\cdot  \left( {\frac{b-a}{4}} \right)^k\left\{{ \left[ {f^{\left(
{k - 1} \right)} \left( a \right) + \left( { - 1} \right)^{k + 1}
f^{\left( {k - 1} \right)} \left( b \right)} \right] }\right.
\\
\left.{+ \left( {1 + \left( { - 1} \right)^{k + 1} } \right)
f^{\left( {k - 1} \right)} \left( {\frac{{a + b}}{2}} \right)
}\right\},
\end{multline*}
and
\begin{align*}
k\left( {t,\frac{3a+b}{4}} \right) = \left\{ \begin{array}{l}
 t - a,\,\,\,\,\,\,\,\,\,\,\,\,\,\,\,\,\,\, t \in \left[ {a,\frac{3a+b}{4}} \right] \\
 t - \frac{{a + b}}{2},\,\,\,\,\,\,\,\,\,\,\,t \in \left( {\frac{3a+b}{4},\frac{a+3b}{4}} \right) \\
 t - b,\,\,\,\,\,\,\,\,\,\,\,\,\,\,\,\,\,\, t \in \left[ {\frac{a+3b}{4},b} \right] \\
 \end{array} \right..
\end{align*}

$\bullet$ If $x=\frac{a+b}{2}$, then
\begin{align}
\int_a^b {f\left( y \right)dy} =
\mathcal{Q}_n\left({f,\frac{a+b}{2}}\right)+\mathcal{E}_n\left({f,\frac{a+b}{2}}\right),\label{eq5.6}
\end{align}
such that
\begin{align*}
\mathcal{Q}_n\left({f,\frac{a+b}{2}}\right):=\frac{b-a}{n}\left( {
f\left( \frac{a+b}{2} \right)  + \sum\limits_{k = 1}^{n - 1}
{G^{\frac{a+b}{2}}_k } } \right),
\end{align*}
and
\begin{align*}
\mathcal{E}_n\left({f,\frac{a+b}{2}}\right):=-
\frac{1}{{n!}}\int_a^b {\left( {\frac{a+b}{2} - t} \right)^{n - 1}
k\left( {t,\frac{a+b}{2}} \right)f^{\left( n \right)} \left( t
\right)dt},
\end{align*}
where,
\begin{align*}
G^{\frac{a+b}{2}}_k  = \frac{{\left( {n - k} \right)}}{{k!}} \cdot
\left( {\frac{b-a}{2}} \right)^k   \left[ {f^{\left( {k - 1}
\right)} \left( a \right) + \left( { - 1} \right)^{k + 1}
f^{\left( {k - 1} \right)} \left( b \right)} \right],
\end{align*}
and
\begin{align*}
k\left( {t,\frac{a+b}{2}} \right) = \left\{ \begin{array}{l}
 t - a,\,\,\,\,\,\,\,\,\,\,\,\,\,\,\,\,\,\, t \in \left[ {a,\frac{a+b}{2}} \right] \\
 t - b,\,\,\,\,\,\,\,\,\,\,\,\,\,\,\,\,\,\, t \in \left[ {\frac{a+b}{2},b} \right] \\
 \end{array} \right..
\end{align*}

A general quadrature rule via harmonic sequence of polynomials can
be considered as follows:
\begin{align}
\int_a^b {f\left( y \right)dy} =
\mathcal{Q}_n\left({f,P_n,x}\right)+\mathcal{E}_n\left({f,P_n,x}\right),\qquad
\forall x\in \left[a,\frac{a+b}{2}\right]\label{eq5.7}
\end{align}
where $\mathcal{Q}_n\left({f,P_n,x}\right)$ is the quadrature
formula given by
\begin{align*}
\mathcal{Q}_n\left({f,P_n,x}\right):=\frac{b-a}{n}\left[
{\frac{{f\left( x \right) + f\left( {a + b - x} \right)}}{2} +
\sum\limits_{k = 1}^{n - 1} { \left\{{ T_k \left(x\right)+
\widetilde{F_k } \left(a,b\right) }\right\}} } \right],
\end{align*}
with error term
\begin{align*}
\mathcal{E}_n\left({f,P_n,x}\right):=- \frac{{\left( { - 1}
\right)^{n - 1} }}{{ n}}\int_a^b {P_{n - 1} \left( t
\right)S\left( {t,x} \right)f^{\left( n \right)} \left( t
\right)dt},
\end{align*}
such that
\begin{align*}
T_k  \left(x\right)= \frac{\left( { - 1} \right)^k}{2} \left\{
{P_k \left( x \right)f^{\left( k \right)} \left( x \right) + P_k
\left( {a + b - x} \right)f^{\left( k \right)} \left( {a + b - x}
\right)} \right\}
\end{align*}
and
\begin{align*}
\widetilde{F_k }\left( {a,b} \right) = \frac{{\left( { - 1}
\right)^k \left( {n - k} \right)}}{{b - a}}\left[ {P_k \left( a
\right)f^{\left( {k - 1} \right)} \left( a \right) - P_k \left( b
\right)f^{\left( {k - 1} \right)} \left( b \right)} \right].
\end{align*}

\begin{remark}
Identities \eqref{eq5.1} and \eqref{eq5.7} can be considered as
Euler--Maclaurin type formulae for symmetric values.
\end{remark}

\begin{remark}
As we mentioned at the end of intoduction section,  many authors
used some key or general expansion formulas such as
Euler--Maclaurin type formulae and Bernoulli polynomials (\em{cf.}
\cite{Dedic2}) to construct some quadrature rules of Newton--Cotes
and Gauss types as done in Franji\'{c} works \cite{F1}--\cite{F5}.
Our expansions, the identities \eqref{eq5.1} and \eqref{eq5.7} can
be considered as general key  formuals instaed of those used in
\cite{F1}--\cite{F5} to construct several quadrature formulas for
an arbitrary $n$-th differentiable real function. The same remark
holds for the formuals \eqref{eq1.5}, \eqref{eq1.9} and
\eqref{eq1.10}.
\end{remark}

\subsection{Errors bounds via  Chebyshev-Gr\"{u}ss type
inequalities}

In what follows, error bounds for the quadrature rules obtained in
Section \ref{sec5} are proved. The proof of these bounds can be
deduced from Corollary \ref{cor2} and Corollary \ref{cor6}.

\begin{proposition}
Let $I$ be a real interval, $a,b \in I^{\circ}$ $(a<b)$. Let $f:I
\to \mathbb{R}$ be such that $f$ is $n$-times differentiable
function such that $f^{(n+1)}$ $(n\ge2)$ is absolutely continuous
on $\left[a,b\right]$. If $f^{\left( {n - 1} \right)}\left( {a}
\right)=f^{\left( {n - 1} \right)}\left( {b} \right)$, then for
all $n\ge2$ we have
\begin{multline}
\left|{\mathcal{E}_n\left({f,a}\right)}\right|
\\
\le\left\{
\begin{array}{l}
\frac{1}{2^{n-1}}\left( {\frac{{n - 2}}{n}} \right)^{n - 2}
\frac{{n^2 - 2n + 2}}{{12n\cdot (n!)^2}} \left( {b-a}
\right)^{n+2}\cdot \left\|{f^{\left( {n + 1} \right)}
}\right\|_{\infty} ,\,\,\,\,\,\,\,\,\,{\rm{if}}\,\,f^{(n+1)} \in
L_{\infty}\left(\left[a,b\right]\right)\\
 \\
\left( {\frac{{n - 2}}{n}} \right)^{n - 2} \frac{{n^2  - 2n +
2}}{{4n\cdot (n!)^2}}   \left( {2^{ - n - 2}  - 2^{ - 2n - 2} }
\right) \left( {b -  a} \right)^{n-1}\cdot  \left( {M -  m}
\right),\,\,\,
{\rm{if}}\,\, m\le f^{(n)} \le M, \\
 \\
 \frac{\left( {b-a} \right)^{n +
\frac{3}{2}}}{2^{n - \frac{1}{2}}\cdot  (n!)^2\cdot \pi^2}
B^{\frac{1}{2}}\left( n \right)
 \cdot
\left\|{f^{\left( {n + 1} \right)} }\right\|_{2} ,\,\,\,\,
{\rm{if}}\,\,f^{(n+1)}
\in L_{2}\left(\left[a,b\right]\right),\\
\\
\frac{1}{2^{n-1}}\left( {\frac{{n - 2}}{n}} \right)^{n - 2}
\frac{{n^2 - 2n + 2}}{{8n\cdot (n!)^2}} \left( {b-a}
\right)^{n+1}\cdot \left( {M - m} \right),\,\,\,
{\rm{if}}\,\, m\le f^{(n)} \le M, \\
\\
\left( {\frac{{n - 2}}{n}} \right)^{n - 2} \frac{{n^2  - 2n +
2}}{{8n\cdot (n!)^2}} \left( {2^{ - n - 2}  - 2^{ - 2n - 2} }
\right) \left( {b -  a} \right)^{n+1}\cdot \left\|{f^{\left( {n +
1} \right)} }\right\|_{\infty},\,\,\, {\rm{if}}\,\,f^{(n+1)} \in
L_{\infty}\left(\left[a,b\right]\right),
\end{array} \right.,\label{eq5.8}
\end{multline}
\begin{multline}
\left|{\mathcal{E}_n\left({f,\frac{3a+b}{4}}\right)}\right|
\\
\le\left\{
\begin{array}{l}
\frac{1}{4^{n-1}}\left( {\frac{{n - 2}}{n}} \right)^{n - 2}
\frac{{n^2 - 2n + 2}}{{12n\cdot (n!)^2}} \left( {b-a}
\right)^{n+2}\cdot \left\|{f^{\left( {n + 1} \right)}
}\right\|_{\infty} ,\,\,\,\,\,\,\,\,\,{\rm{if}}\,\,f^{(n+1)} \in
L_{\infty}\left(\left[a,b\right]\right)\\
 \\
\left( {\frac{{n - 2}}{n}} \right)^{n - 2} \frac{{n^2  - 2n +
2}}{{4n\cdot (n!)^2}}   \left( {2^{ - n - 2}  - 2^{ - 2n - 2} }
\right) \left( {b -  a} \right)^{n-1}\cdot  \left( {M -  m}
\right),\,\,\,
{\rm{if}}\,\, m\le f^{(n)} \le M, \\
 \\
\frac{\left( {b-a} \right)^{n + \frac{3}{2}}}{{4^{n -
\frac{1}{2}}\cdot (n!)^2\cdot \pi^2}}
 \sqrt {A\left( n \right) +
B\left( n \right)}
 \cdot
\left\|{f^{\left( {n + 1} \right)} }\right\|_{2} ,\,\,\,\,
{\rm{if}}\,\,f^{(n+1)}
\in L_{2}\left(\left[a,b\right]\right),\\
\\
\frac{\left( {b-a} \right)^{n+1}}{4^{n-1}}\left( {\frac{{n -
2}}{n}} \right)^{n - 2} \frac{{n^2  - 2n + 2}}{{8n\cdot (n!)^2}}
\cdot \left( {M - m} \right),\,\,\,
{\rm{if}}\,\, m\le f^{(n)} \le M, \\
\\
\left( {\frac{{n - 2}}{n}} \right)^{n - 2} \frac{{n^2  - 2n +
2}}{{8n\cdot (n!)^2}} \left( {2^{ - n - 2}  - 2^{ - 2n - 2} }
\right) \left( {b -  a} \right)^{n+1}\cdot \left\|{f^{\left( {n +
1} \right)} }\right\|_{\infty},\,\,\, {\rm{if}}\,\,f^{(n+1)} \in
L_{\infty}\left(\left[a,b\right]\right),
\end{array} \right.\label{eq5.9}
\end{multline}
\begin{multline}
\left|{\mathcal{E}_n\left({f,\frac{a+b}{2}}\right)}\right|
\\
\le\left\{
\begin{array}{l}
\frac{1}{2^{n-1}}\left( {\frac{{n - 2}}{n}} \right)^{n - 2}
\frac{{n^2 - 2n + 2}}{{12n\cdot (n!)^2}} \left( {b-a}
\right)^{n+2}\cdot \left\|{f^{\left( {n + 1} \right)}
}\right\|_{\infty} ,\,\,\,\,\,\,\,\,\,{\rm{if}}\,\,f^{(n+1)} \in
L_{\infty}\left(\left[a,b\right]\right)\\
 \\
\left( {\frac{{n - 2}}{n}} \right)^{n - 2} \frac{{n^2  - 2n +
2}}{{4n\cdot (n!)^2}}  \left( {2^{ - n - 2}  - 2^{ - 2n - 2} }
\right) \left( {b -  a} \right)^{n-1}\cdot  \left( {M -  m}
\right),\,\,\,
{\rm{if}}\,\, m\le f^{(n)} \le M, \\
 \\
 \frac{\left( {b-a} \right)^{n +
\frac{3}{2}}}{2^{n - \frac{1}{2}}\cdot  (n!)^2\cdot \pi^2}
A^{\frac{1}{2}}\left( n \right)
 \cdot
\left\|{f^{\left( {n + 1} \right)} }\right\|_{2} ,\,\,\,\,
{\rm{if}}\,\,f^{(n+1)}
\in L_{2}\left(\left[a,b\right]\right),\\
\\
\frac{1}{2^{n-1}}\left( {\frac{{n - 2}}{n}} \right)^{n - 2}
\frac{{n^2 - 2n + 2}}{{8n\cdot (n!)^2}} \left( {b-a}
\right)^{n+1}\cdot \left( {M - m} \right),\,\,\,
{\rm{if}}\,\, m\le f^{(n)} \le M, \\
\\
\left( {\frac{{n - 2}}{n}} \right)^{n - 2} \frac{{n^2  - 2n +
2}}{{8n\cdot (n!)^2}} \left( {2^{ - n - 2}  - 2^{ - 2n - 2} }
\right) \left( {b -  a} \right)^{n+1}\cdot \left\|{f^{\left( {n +
1} \right)} }\right\|_{\infty},\,\,\, {\rm{if}}\,\,f^{(n+1)} \in
L_{\infty}\left(\left[a,b\right]\right),
\end{array} \right.,\label{eq5.10}
\end{multline}
where
\begin{align*}
A\left( n \right) = \frac{{2\left( {n - 1} \right)^2 }}{{\left(
{2n - 1} \right)\left( {2n - 2} \right)\left( {2n - 3}
 \right)}}
\end{align*}
and
\begin{align*}
B\left( n \right) = \frac{{2^{2n - 3} \left( {2n - 1}
\right)\left( {2n - 2} \right) + 4n\left( {2n - 1} \right) + 2n^2
}}{{\left( {2n - 1} \right)\left( {2n - 2} \right)\left( {2n - 3}
\right)}}, \qquad \forall n \ge2.
\end{align*}
And finally
\begin{multline}
\left|{\mathcal{E}_n\left({f,P_n,x}\right)} \right|
\\
\le\left\{
\begin{array}{l}
\frac{\left( {b - a} \right)^3}{12n}\left\| {P_{n - 1}  + P_{n -
2} S\left( { \cdot ,x} \right)} \right\|_\infty \cdot
\left\|{f^{\left( {n + 1} \right)} }\right\|_{\infty}
,\,\,\,\,\,\,\,\,\,\,{\rm{if}}\,\,f^{(n+1)} \in
L_{\infty}\left(\left[a,b\right]\right)\\
 \\
\frac{b-a}{4n}\left( {M_1  - m_1 } \right)\left( {M_2  - m_2 }
\right),\,\,\,\,\,\,\,\,\,\,\,\,\,\,\,\,\,\,\,\,\,\,\,\,\,\,\,\,\,\,\,\,\,\,\,\,\,\,\,\,\,\qquad
{\rm{if}}\,\, m_1\le f^{(n)} \le M_1, \\
 \\
\frac{\left( {b - a} \right)^2}{\pi^2n}D\left( {n,x} \right)
\cdot \left\|{f^{\left( {n + 1} \right)} }\right\|_{2}
,\qquad\qquad\qquad\qquad\,\,\,\,\,\, {\rm{if}}\,\,f^{(n+1)}
\in L_{2}\left(\left[a,b\right]\right),\\
\\
\frac{\left( {b - a} \right)^2}{8n}\left\| {P_{n - 1}  + P_{n - 2}
S\left( { \cdot ,x} \right)} \right\|_\infty \cdot    \left( {M_1
- m_1}
\right),\,\,\,\,\,\,\,\,\,{\rm{if}}\,\, m_1\le f^{(n)} \le M_1, \\
\\
\frac{\left( {b - a} \right)^2}{8n}\left( {M_2  - m_2 }
\right)\cdot \left\|{f^{\left( {n + 1} \right)}
}\right\|_{\infty},\,\,\,\,\,\,\qquad\qquad\qquad
{\rm{if}}\,\,f^{(n+1)} \in
L_{\infty}\left(\left[a,b\right]\right),
\end{array} \right.\label{eq5.11}
\end{multline}
where $D\left( {n,x} \right)$ and $m_2,M_2$ are defined in Theorem
\ref{thm11}.
\end{proposition}

\begin{remark}
Other error bounds can be stated using \eqref{eq2.19},
\eqref{eq4.11} and \eqref{eq4.12}.
\end{remark}

\centerline{}

\centerline{}

\end{document}